\documentclass[aihp,preprint]{imsart}

\RequirePackage{amssymb, amsmath, amsfonts,amsthm}
\RequirePackage[numbers]{natbib}
\RequirePackage[colorlinks,citecolor=blue,urlcolor=blue]{hyperref}
\usepackage{enumitem} 
\usepackage{bbm} 

\arxiv{arXiv:1702.01091}

\startlocaldefs
\numberwithin{equation}{section}
\theoremstyle{plain}
\newtheorem{thm}{Theorem}[section]
\newtheorem{cor}[thm]{Corollary}
\newtheorem{lem}[thm]{Lemma}
\newtheorem{prop}[thm]{Proposition}
\theoremstyle{remark}
\newtheorem{rem}{Remark}[section]

\newtheorem{qu}[rem]{Open question}
\theoremstyle{definition}
\newtheorem{dfn}{Definition}[section]

\newcommand{\e}{\mathrm{e}} 
\newcommand{\expp}[1]{\exp\left( #1 \right) } 
\newcommand{\N}{\mathbb{N}}
\newcommand{\R}{\mathbb{R}}
\newcommand{\U}{\mathbb{U}} 

\newcommand\multiset[1] { \{\!\!\{  #1  \}\!\!\} } 
\newcommand\dd{\mathrm{d}}
\newcommand{\Exp}[2][]{\mathbb{E}_{#1}\left[ #2 \right]}
\newcommand{\EE}{\mathbb{E}} %

\newcommand{\Expcond}[2]{\mathbb{E}\left[ \left. #1 ~\right|~ #2 \right]}
\newcommand{\PP}{\mathbb{P}} 
\newcommand{\Prob}[2][]{\mathbb{P}_{#1}\left( #2 \right)}

\newcommand{\ind}[1]{\mathbbm{1}_{\left\{ #1 \right\} } } 

\newcommand{\F}{\mathcal{F}}

\newcommand{\Lc}{\mathcal{L}} 
\newcommand{\prm}[2]{\left\langle #1, #2 \right\rangle }
\newcommand{\Cinfty}{\mathcal C^{\infty}_c(\R_+)} 

\newcommand{\eqdis}{\overset{d}{=}}

\newcommand{\Sd}{\mathcal{S}} 
\newcommand{\Sdtwo}{\mathcal{S} \setminus \mathcal{S}_1} 
\newcommand{\sd}{\mathbf s} 
\newcommand{\Rd}{\mathcal{R}} 
\newcommand{\Rdone}{\mathcal{R}_1} 
\newcommand{\Rdtwo}{\mathcal{R} \setminus \mathcal{R}_1} 
\newcommand{\rr}{\mathbf{r}} 
\newcommand{\co}{c_o^{\downarrow}} 

\newcommand{\dOU}{\theta} 
\newcommand{\levOU}{\ell} 
\newcommand{\cutOU}[2][\levOU]{ #2^{(#1)}}
\newcommand{\cutOUd}[2][\levOU]{ #2^{(#1)}}

\newcommand{\dom}{\mathrm{dom}(\kappa)} 

\newcommand{\partialX}{(0, 0, \ldots) } %
\newcommand{\Xc}{\mathcal{X}} 
\newcommand{\Pm}{\mathrm{P}} %
\newcommand{\Em}{\mathrm{E}} %
\newcommand{\Ps}{\mathbf{P}} 
\newcommand{\Xs}{\mathbf{X}} 
\newcommand{\Xsd}{\mathbf{X}} 
\newcommand{\xs}{\mathbf{x}} %
\newcommand{\Zs}{\mathcal{Z}} 

\newcommand{\lp}[1][p]{\ell^{#1 \downarrow}} 

\newcommand{\tb}{\tau} 

\endlocaldefs

\begin{document}

\begin{frontmatter}

\title{A growth-fragmentation model related to Ornstein-Uhlenbeck type processes}

\runtitle{Ornstein-Uhlenbeck type growth-fragmentation processes}

\begin{aug}

\author{\fnms{Quan} \snm{Shi}\thanksref{a}\ead[label=e1]{quanshi.math@gmail.com}}
\address[a]{University of Oxford, Oxford, OX1 3LB, UK.\\
\printead{e1}}

\runauthor{Q.Shi}

\affiliation{University of Oxford}

\end{aug}

\begin{abstract}
Growth-fragmentation processes describe systems of particles in which each particle may grow larger or smaller, and divide into smaller ones as time proceeds. Unlike previous studies, which have focused mainly on the self-similar case, we introduce a new type of growth-fragmentation which is closely related to L\'evy driven Ornstein-Uhlenbeck type processes. Our model can be viewed as a generalization of compensated fragmentation processes introduced by Bertoin, or the stochastic counterpart of a family of growth-fragmentation equations. We establish a convergence criterion for a sequence of such growth-fragmentations. We also prove that, under certain conditions, this system fulfills a law of large numbers. 
\end{abstract}

\begin{abstract}[language=french]
Les processus de croissance-fragmentation \'etudient l'\'evolution au cours du temps de syst\`emes de particules, dans lesquels la taille de chaque particule peut cro\^itre et d\'ecro\^itre, les particules pouvant parfois se fragmenter. Contrairement aux \'etudes pr\'ec\'edentes, qui se sont concentr\'ees principalement sur les cas auto-similaires, nous introduisons un nouveau mod\`ele qui est associ\'e aux processus d'Ornstein-Uhlenbeck li\'es aux processus de L\'evy. Notre mod\`ele peut \^etre vu comme une g\'en\'eralisation des processus de fragmentation compens\'es introduits par Bertoin, ou la contrepartie stochastique d'une famille d'\'equations de croissance-fragmentation. Nous \'etablissons un crit\`ere de convergence pour une suite de telles croissance-fragmentations, et une loi des grands nombres dans un cas particulier. 
\end{abstract}

\begin{keyword}[class=MSC]
60G51, 60J80
\end{keyword}

\begin{keyword}
\kwd{growth-fragmentation}
\kwd{Ornstein-Uhlenbeck type process}
\kwd{branching particle system}
\kwd{law of large numbers}
\end{keyword}



\end{frontmatter}

\section{Introduction}
Fragmentation processes describe the evolution of particle systems in which each particle may split randomly into smaller ones as time passes, independently of the others; see \cite{Bertoin:Frag-Book} for a comprehensive overview. Recently, Bertoin \cite{Bertoin:CF,Bertoin:GF-Markovian} extended fragmentations to \emph{growth-fragmentation processes}, by allowing the size of each particle to increase or decrease gradually.  
In both (pure) fragmentations and growth-fragmentations, previous research has mainly focused on the \emph{self-similar} case, which means that the particle system behaves in the same way when viewed at different scales on space and time.

In the present work, we propose a new type of growth-fragmentation that possesses a different scaling property, to be given shortly. We name it an \emph{Ornstein-Uhlenbeck (OU) type growth-fragmentation process}. Informally speaking, our model describes a particle system in which the mass of each particle evolves according to the exponential of an \emph{OU type process} $(Z(t), t\geq 0)$ driven by a L\'evy process $\xi$: \begin{equation}\label{Ch4:eq:OU1}
  Z(t) := \e^{-\dOU t} Z(0) + \int_0^t \e^{-\dOU (t-s)} \dd \xi(s),\qquad t\geq 0,
 \end{equation}
where $\dOU\in \R$ and the integral is defined in the sense of a stochastic integral, as the L\'evy process $\xi$ is a semimartingale. If $\xi$ is a Brownian motion, then $Z$ is a well-known Gaussian OU process.
Furthermore, each particle splits randomly into smaller ones according to a \emph{dislocation measure} $\nu$, which is a sigma-finite measure on the \emph{mass-partition space}
\begin{equation}\label{eq:Sd}
 \Sd :=  \left\{\sd := (s_1, s_2, \ldots):~  s_1\geq s_2 \geq \ldots \geq 0, \sum_{i=1}^{\infty}s_i\leq 1  \right\},
 \end{equation}
that satisfies 
\begin{equation}\label{Ch4:eq:nu}
 \nu({(1,0,\ldots)})=0 \quad \text{and}\quad \int_{\Sd} (1-s_1)^2 \nu(\dd \sd) <\infty.
\end{equation}
 For every $\sd =(s_1, s_2, \ldots)\in \Sd$, a particle with mass $x>0$ splits at rate $\nu(\dd \sd)$ into a sequence of particles with masses $(xs_1, xs_2, \ldots)$. Each child fragment continues to evolve in a similar way, independently of the others. 
For $t\ge 0$, let 
\[\Xsd(t):=(X_1(t), X_2(t), \ldots )\]
denote the decreasing sequence of the masses of particles alive at time $t$. Then 
$\Xsd  := \Big(\Xsd(t), t\geq 0\Big)$ is \emph{an OU type growth-fragmentation process}. 
The precise definition of our model is given in Section~\ref{Ch4:sec:construction}.

 Let $\co$ be the space of decreasing null sequences (that converge to $0$), endowed with the $\ell^{\infty}$-norm. 
We prove that our process $\Xsd$ is a $\co$-valued Markov process which possesses a c\`adl\`ag version, and moreover satisfies the following two properties. For every $x\in\R_+ := (0,\infty)$, let $\Ps_{x}$ denote  the law of $\Xsd$ with initial value $\Xsd(0)=(x,0,\ldots)\in \co$. 
\begin{enumerate}[label=(\textbf{P\arabic{*}}), ref=(\textbf{P\arabic{*}}),leftmargin=3.0em]
\item \label{Ch4:P1} (Branching property) For every sequence $\mathbf{x} = (x_1,x_2, \ldots) \in \co$, the process $\Xsd$ starting from $\Xsd(0)= \mathbf{x}$ has the same law as the union of the elements, arranged in decreasing order, of a family of independent OU type growth-fragmentations $(\Xs^{[i]})_{i\geq 1}$, where each $\Xs^{[i]}$ has distribution $\Ps_{x_i}$.
\item \label{Ch4:P2}  (OU property) With $\dOU\in \R$ being the index appeared in \eqref{Ch4:eq:OU1},
for every $x \in \R_+$, the distribution of the rescaled process $(x^{\expp{-\dOU t}} \Xsd (t) )_{t\geq 0}$ under $\Ps_{1}$ is $\Ps_{x}$. 
\end{enumerate}
The branching property indicates that the fragments evolve independently from one another. The OU property is due to the scaling property of the exponential of an OU type process (a direct consequence of \eqref{Ch4:eq:OU1}). 
For comparison, note that a \emph{self-similar growth-fragmentation} $\mathbf{Y}$ (including the case of pure fragmentations) fulfills the same branching property, but a different scaling property: for a certain index $\alpha\in \R$, the rescaled process $(x \mathbf{Y} (x^{\alpha}t) )_{t\geq 0}$ under $\Ps_{1}$ is $\Ps_{x}$; see \cite[Theorem~2]{Bertoin:GF-Markovian} and \cite[Definition~2]{Berestycki:Frag-Ranked}. 
The special case $\theta=0$ of our model coincides with \emph{homogeneous} growth-fragmentations (self-similar with $\alpha=0$). 

Our model is partially motivated by \cite{BertoinBaur} (see also a related work \cite{Moehle:MittagLeffler}), results in which imply that a certain OU type growth-fragmentation naturally arises in dynamical percolation on an infinite recursive tree; see Section~\ref{Ch4:sec:RRT} for details. Besides this motivation, our model may have potential applications, as OU type processes are widely applied in various domains: in biology, they are used in a neuronal model with signal-dependent noise \cite{Lansky:OU-neuronal}; in finance, they are used in an option price model with stochastic volatility \cite{NielsenShephard:1,NielsenShephard:2}, to name just a few.

Since the dislocation measure $\nu$ is allowed to be infinite, branching events can occur with an infinite intensity. Due to this fact, the construction of our model is subtle. Our approach relies on  a truncation procedure introduced by Bertoin~\cite{Bertoin:CF} to build homogeneous growth-fragmentations (which he called \emph{compensated fragmentation processes}). 
Specifically, if we discard the small (in size at birth relative to their parent) fragments, then the truncated process has a finite branching rate, which can be easily built with a genealogical structure. We finally re-incorporate the small fragments by considering the increasing limit. The technical difficulty in adopting this approach is that one needs to check that such a growth-fragmentation does not \emph{locally explode}, that is, for every $x>0$, only a finite number of fragments have size greater than $x$ at every time. This is justified by Theorem~\ref{Ch4:thm:mean}, which relies crucially on the integrability assumption \eqref{Ch4:eq:nu}. See \cite{BertoinStephenson} for a related construction of binary self-similar growth-fragmentations.

It is sometimes more convenient to work with the logarithmic transform of a growth-fragmentation. 
After logarithmic transformation, homogeneous (pure) fragmentations can be viewed as continuous time branching random walks \cite{BertoinRouault}, and homogeneous growth-fragmentations are related with 
branching L\'evy processes \cite{Bertoin:CF}. 
In line with these observations, we first introduce an \emph{OU type branching Markov process} (Definition~\ref{Ch4:dfn:BOUP}), which is similar to a branching random walk, but with a spatial motion given by an OU type process. 
An OU type growth-fragmentation process is just the exponential of an OU type branching Markov process. 
Both the truncation procedure and the non-explosion property mentioned above are established for 
OU type branching Markov processes in Section~\ref{Ch4:sec:BOUP}.

We obtain two major results. We first prove (Theorem~\ref{thm:CV}) the convergence of a sequence of OU type growth-fragmentations when their characteristics converge in some sense. This conclusion generalizes \cite[Theorem~2]{Bertoin:CF}. The other result (Corollary~\ref{Ch4:cor:LLN}) concerns the long-time asymptotic behavior of OU type growth-fragmentations. Roughly speaking, we show, for a particular case, that the (random) empirical measure of particle sizes converges in probability to a deterministic measure. 
This law of large numbers should be compared with the limit theorems for self-similar fragmentations and growth-fragmentations \cite{BertoinGnedin, Dadoun:Asymptotics}, as well as the law of large numbers in the context of branching Gaussian OU processes \cite{AdamczakMilos:CLT} and branching diffusions \cite{EHK:SLLN}. 
 
We also find (Proposition~\ref{Ch4:prop:MOU}) that OU type growth-fragmentations bear a connection with Bertoin's \emph{Markovian growth-fragmentations} \cite{Bertoin:GF-Markovian} and (Proposition~\ref{Ch4:prop:gf-eq}) that they are the stochastic counterparts of a family of (deterministic) growth-fragmentation equations; see \cite{BertoinWatson, DoumicEscobedo, Doumic,MischlerScher} for related works on the latter topic.

The paper is organized as follows. 
Section~\ref{Ch4:sec:BOUP} introduces OU type branching Markov processes. 
Section~\ref{Ch4:sec:OUGFP} studies OU type growth-fragmentations in depth. 
After giving the construction in Section~\ref{Ch4:sec:construction}, 
we present a many-to-one formula and related growth-fragmentation equations in Section~\ref{Ch4:sec:mto}, establish a convergence criterion for a sequence of OU type growth-fragmentations in Section~\ref{Ch4:sec:CV}, and prove a law of large numbers in Section~\ref{Ch4:sec:LLN}. 
Section~\ref{Ch4:sec:Markovian} draws connections to Markovian growth-fragmentations \cite{Bertoin:GF-Markovian}. 
Finally, Section~\ref{Ch4:sec:RRT} offers a remarkable example related to a destruction process of infinite random recursive tree \cite{BertoinBaur}.

\section{OU type branching Markov processes}\label{Ch4:sec:BOUP}
In this section, we first recall some background on OU type processes, and then introduce OU type branching Markov processes.

\subsection{Preliminaries: Ornstein-Uhlenbeck type processes}\label{Ch4:sec:OU}

Let us present some elementary background on Ornstein-Uhlenbeck (OU) type processes driven by L\'evy processes; see \cite{Applebaum:Levy} or \cite[Section~17]{Sato:Levy}. We also refer to \cite{Bertoin:Levy} for properties of L\'evy processes. Implicitly, throughout this work we only consider OU type processes without positive jumps.

Let $\xi$ be a L\'evy process without positive jumps, possibly killed, which is often referred to as a \emph{spectrally negative L\'evy process}. It is characterized by its Laplace exponent $\Phi: [0,\infty) \to \R$ such that
\[\Exp{\e^{q \xi(t)}} = \e^{\Phi(q) t}, \qquad \text{ for all } t, q\geq 0. \]
The function $\Phi$ is continuous and convex on $[0,\infty)$. Furthermore, it is given by the L\'evy-Khintchine formula 
\begin{equation}\label{Ch4:eq:Phi}
\Phi(q) =-k + \frac{1}{2} \sigma^2 q^2 + c q + \int_{(-\infty,0)} \big( \e^{q y}-1 + q (1 - \e^{y})  \big) \Lambda  (\dd y), \qquad q\geq 0,
\end{equation}
where $k\geq 0$, $\sigma \geq 0$, $c \in \R$, and the L\'evy measure $\Lambda$ on $(-\infty, 0)$ satisfies 
\begin{equation}\label{Ch4:eq:mLevy}
 \int_{(-\infty,0)} (|y|^2 \wedge 1) \Lambda (\dd y) <\infty.
\end{equation}
We say that $\xi$ has characteristics $(\sigma, c , \Lambda, k)$. In the L\'evy-Khintchine formula, we can also replace $q (1 - \e^{y})$ in the integral by $-q y \ind{y>-1}$, as often in the literature, then we need to change the drift coefficient $c$. 

Let $\dOU\in \R$, we next define an \emph{Ornstein-Uhlenbeck (OU) type process} $Z$ with characteristics $(\sigma, c , \Lambda, k, \dOU)$ or simply $(\Phi, \dOU)$, starting from $Z(0)= z\in \R$, by
\begin{equation}\label{Ch4:eq:solOU}
  Z(t) = \e^{-\dOU t} z + \int_0^t \e^{-\dOU (t-s)} \dd \xi(s), \qquad t\geq 0.
 \end{equation}
By convention, if $\xi$ is killed at time $\zeta\geq 0$, then $Z(t):= -\infty$ for every $t\geq \zeta$.  When $\dOU>0$, $Z$ is called an \emph{inward} OU type process; respectively, while $\dOU<0$, $Z$ is called an \emph{outward} OU type process.
Note that in the literature, OU type processes often only refer to the inward case ($\dOU>0$). 
Furthermore, it is well-known (\cite[equation (17.2) and Lemma~17.1]{Sato:Levy}) that $Z$ is the pathwise unique solution of the stochastic integral equation 
\begin{equation*}
Z(t) = z + \xi(t) - \dOU \int_0^t Z(s) \dd s,
\end{equation*}
and that there is 
\begin{equation}\label{Ch4:eq:Lap}
\Exp{\exp\big(q Z(t)\big)}= \expp{\e^{-\dOU t} zq + \int_0^t \Phi(q \e^{-\dOU s} ) \dd s }, \qquad \text{for all } t,q \geq 0.
\end{equation}

 The next observation follows plainly from \eqref{Ch4:eq:solOU}. 
\begin{lem}\label{Ch4:lem:addOU}
	If $Z_1$ and $Z_2$ are independent OU type processes with respective characteristics $(\Phi_1, \dOU)$ and $(\Phi_2, \dOU)$, then $Z_1+Z_2$ is an OU type process with characteristics $(\Phi_1+ \Phi_2, \dOU)$. 
\end{lem}

Under certain conditions, an \emph{inward} OU type process converges in distribution to its invariant probability distribution. 
\begin{lem}[{\cite[Theorem~17.5 and 17.11]{Sato:Levy}}]\label{Ch4:lem:inv}
	If $\theta>0$ and $\Lambda$ satisfies 
	\begin{equation}\label{Ch4:eq:invH}
	\int_{(-\infty,- \log 2)} \log |y| ~\Lambda(d y) <\infty,  
	\end{equation}
	then the OU type process $Z$ possesses a unique invariant probability distribution $\Pi$, which is a probability measure on $\R$ with Laplace transform
	\begin{equation*}
	\int_{\R} \e^{qy} \Pi(\dd y) = \expp{\int_0^{\infty}\Phi(\e^{-\dOU s} q) \dd s}, \qquad q\geq 0.
	\end{equation*}
	Moreover, for every bounded and continuous function $g : \R\to \R$ there is 
	\begin{equation*}
	\lim_{t\to\infty}  \EE\big[ g (Z(t))\big]  = \int_{\R} g(y) \Pi(\dd y).
	\end{equation*}
	If \eqref{Ch4:eq:invH} does not hold, then $Z$ does not have any invariant probability distribution. 
\end{lem}

We remark that the invariant probability distribution $\Pi$ is \emph{self-decomposable}, which means that if a random variable $Y$ has law $\Pi$, then for every constant $r\in (0,1)$, there exists an independent random variable $Y^{(r)}$, such that $Y\eqdis r Y + Y^{(r)}$. Conversely, every self-decomposable measure is the invariant probability distribution of a certain OU type process (with possibly positive jumps). See \cite[Definition~15.1 and Theorem~17.5]{Sato:Levy} for details. 

\subsection{OU type branching Markov processes with finite birth-intensity}\label{Ch4:sec:BOUW}

With convention $\e^{-\infty}:=0$, introduce the space 
\[ \Rd :=  \left\{\rr = (r_1, r_2, \ldots)~:~ 0\geq r_1\geq r_2 \geq \ldots\geq -\infty,~ \sum_{i=1}^{\infty}\e^{r_i}\leq 1  \right\},\]
Let $\mu$ be a sigma-finite measure on $\Rd$ that satisfies 
\begin{equation}\label{Ch4:eq:mu}
  \mu( \{(0, -\infty, -\infty,\ldots)\} )=0 \quad \text{and}\quad  \int_{\Rd} (1-\e^{r_1})^2 \mu(\dd \rr) <\infty. 
\end{equation}  
Define $\#\rr:= \sup\{i\ge 1~\colon~ r_i >-\infty\}$, with convention $\sup \emptyset = 0$,  and 
\[
\Rdone :=  \left\{ \rr\in \Rd ~\colon~\#\rr = 1  \right\}.
\]
Then \eqref{Ch4:eq:mu} ensures that the image of the restriction $\mu|_{\Rdone}$, via the map $\rr \to r_1$ from $\Rdone$ to $(-\infty, 0)$, is a L\'evy measure (that satisfies \eqref{Ch4:eq:mLevy}), which shall be denoted by $\Lambda_1$. 

Informally speaking, an \emph{OU type branching Markov process} describes the positions of the atoms in the following system. Initially, there is a single atom at the origin.  
This atom evolves according to a certain OU type process $Z$ whose L\'evy measure is $\Lambda_1$. 
  The branching mechanism is given by $\mu|_{\Rdtwo}$. Specifically, a particle at any position $y\in \R$ splits into two or more particles at $y+ \rr$ with rate $\mu|_{\Rdtwo}(d \rr)$, and for every $i\in \N$, the child born at position $y + r_i$ evolves according to the law of $Z$ (with starting point $Z(0) = y+ r_i$),  independently of the others. Each child particles continues to branch in a similar way. 
Recall from the introduction that such processes are tailored for the purpose to study growth-fragmentations; note that the space $\Rd$ is obtained by the logarithmic transform from the mass-partition space $\Sd$ defined in \eqref{eq:Sd}, and \eqref{Ch4:eq:mu} is in line with \eqref{Ch4:eq:nu}. 

If we also suppose that 
\begin{equation}\label{eq:fbi}
\int_{\Rdtwo} \# \rr~ \mu(\dd \rr)   <\infty, 
\end{equation}
then the OU type branching Markov process is said to have \emph{finite birth-intensity}, and can be constructed as a marked \emph{Ulam-Harris tree} $\U:= \bigcup_{n=0}^{\infty} \mathbb{N}^n$, with $\N:= \{1,2,3,\ldots\}$ and $\mathbb{N}^0 := \{\emptyset\}$ by convention. An element $u\in \U$ is a finite sequence of natural numbers $u = (n_1, \ldots ,n_{|u|})$, where $|u|\in \N$ stands for the generation of $u$. Write $u_- = (n_1, \ldots ,n_{|u|-1})$ for her mother and $uk = (n_1, \ldots n_{|u|} ,k)$ for her $k$-th daughter with $k\in \N$. The following construction is similar to that of  a branching L\'evy process \cite[Definition~1]{Bertoin:CF}. Notice that combining \eqref{Ch4:eq:mu} and \eqref{eq:fbi} yields
\[
\mu(\Rdtwo)\le \int_{\Rdtwo} \# \rr~ \mu(\dd \rr)  + \mu(\{(-\infty, -\infty,\ldots)\}) \le \int_{\Rdtwo} \# \rr~ \mu(\dd \rr) + \int_{\Rd} (1-\e^{r_1})^2 \mu(\dd \rr) < \infty. 
\] 

\begin{dfn}\label{Ch4:dfn:BOUPF}
Let $\dOU\in \R$, $\sigma \geq 0$, $c\in \R$, and $\mu$ be a sigma-finite measure in $\Rd$ such that \eqref{Ch4:eq:mu} and \eqref{eq:fbi} hold. Consider three independent families $(\lambda_u)_{u \in \U}$, $(Z_u)_{u \in \U}$ and $(\Delta a_{ui}, i\in \N)_{u \in \U}$: 
\begin{itemize}
\item $(\lambda_u)_{u \in \U}$ is a family of i.i.d. exponential variables with parameter $\mu (\Rd \backslash \Rd_1)$. 
\item $(Z_u)_{u \in \U}$ is a family of i.i.d. OU type processes, starting from $Z_u(0)=0$, with characteristics $(\psi,\dOU)$, where 
\begin{equation}\label{Ch4:eq:Psi}
\psi(q):= \frac{1}{2} \sigma^2 q^2 + \bigg( c + \int_{\Rdtwo} (1-\e^{r_1})\mu(\dd \rr) \bigg) q 
+ \int_{\Rd_1}  \left(\e^{qr_1}-1 +  q(1-\e^{r_1}) \right) \mu(\dd \rr), \qquad q\geq 0. 
\end{equation} 
\item $(\Delta a_{ui}, i\in \N)_{u \in \U}$ is a family of i.i.d. sequences, each sequence being distributed according to the conditional probability $\mu (\cdot ~|~ \Rdtwo) $.
\end{itemize}
With initial values $b_{\emptyset}=0$ and $a_{\emptyset}=0$, we define recursively
\[ a_{ui} := \e^{-\dOU \lambda_u}a_u+ Z_u(\lambda_u) + \Delta a_{ui}, \quad  b_{ui} := b_u+\lambda_{u} , \qquad \text{for every } u\in \U, i\in \N.\]
For every $u\in \U$ the triple $(a_u, b_u, \lambda_u)$ stands for the position at birth, the birth time and the lifetime respectively of the particle indexed by $u$. 
Note that if $\Delta a_{ui} = -\infty$, then by convention $a_{ui}: = -\infty$, which means that the atom $ui$ (as well as its descendants) is not taken into account.
This particle moves according to $(\e^{-\dOU r }a_u + Z_u(r))_{r\geq 0}$, which has the law of $Z$ with $Z(0)= a_u$ by \eqref{Ch4:eq:solOU}. 
Then the positions of the particles alive at time $t\geq 0$ form a multiset (that allows multiple instances of its elements)
\[\Zs (t):= \multiset{\e^{-\dOU (t-b_u)}a_u + Z_u(t-b_u):~ u\in \U, b_u \leq t < b_u+ \lambda_u} , \quad t\geq 0.\]
The process $\Zs$ is called an \emph{OU type branching Markov process} with (finite birth-intensity and) characteristics $(\sigma, c, \mu, \dOU)$.  
\end{dfn}

\begin{rem}
One can view a multiset $\mathcal I$ as a point measure $\sum_{i\in \mathcal I} \delta_i$, where $\delta$ stands for the Dirac mass. 
\end{rem}

The term $\int_{\Rdtwo} (1-\e^{r_1})\mu(\dd \rr)$ in \eqref{Ch4:eq:Psi}, which is an analogue of the compensation term in the L\'evy-Khintchine formula \eqref{Ch4:eq:Phi}, is used to compensate for the negative jumps in the branching events induced by $\mu|_{\Rdtwo}$. We place it there for the following purposes. First, this is consistent with \cite[Definition~1]{Bertoin:CF}. So for the case $\dOU = 0$,  an OU type branching Markov process with characteristics $(\sigma^2, c, \mu, 0)$ is a \emph{branching L\'evy process} with characteristics $(\sigma^2, c, \mu)$. 
Second, this induces an important embedding property that we shall now present. For each $\levOU\geq 0$, we cut an OU type branching Markov process $\Zs$ with characteristics $(\sigma, c, \mu, \dOU)$ at level $\levOU$, by keeping at each dislocation the child particle which is the closest to the parent, and by suppressing the other child particles if and only if its distance to the position of the parent at death is larger than or equal to $\levOU$. Let $B(\levOU)\subset \U$ be the set of individuals that are killed by this cutting operation, so $u=(u_1, \ldots, u_{|u|}) \in B(\levOU)$ if and only if
\begin{equation*}
\Delta a_{u_1, \ldots, u_j} \leq -\levOU \text{ and } u_j \geq 2 \text{ for some } j = 1, \ldots ,|u|.
\end{equation*}
For every $r\in[-\infty, 0]$, set
\begin{equation*}
  \cutOU{r}:=\begin{cases}
    r  & \text{ if } r >-\levOU, \\
    -\infty & \text{otherwise}.
  \end{cases}
\end{equation*}
Then for every $\rr=(r_1, r_2, r_3, \ldots) \in \Rd$, we define 
\begin{equation}\label{Ch4:eq:cut}
  \cutOU{\rr}:= (r_1, \cutOU{r_2}, \cutOU{r_3}, \ldots).
\end{equation}
 Let $\cutOU{\mu}$ be the image of $\mu$ by the map $\rr \mapsto \cutOU{\rr}$.

 \begin{lem}[Key embedding property]\label{Ch4:lem:cut}
   The truncated process 
   \begin{equation}\label{Ch4:eq:truncate}
\cutOU{\Zs}(t) := \multiset{\e^{-\dOU (t-b_u)}a_u + Z_u(t-b_u):~ u\in \U, u\not\in B(\levOU), b_u \leq t < b_u+ \lambda_u}, \qquad t\geq 0
   \end{equation}
is an OU type branching Markov process with characteristics $(\sigma, c, \cutOU{\mu}, \dOU)$.
 \end{lem}
It is not difficult to prove Lemma~\ref{Ch4:lem:cut} by using similar arguments to those of an analogous result \cite[Lemma~3]{Bertoin:CF} for branching L\'evy processes. We include the proof of Lemma~\ref{Ch4:lem:cut} in Appendix~\hyperref[sec:A]{A} for the sake of completeness. 
  
 For the particular case with $\levOU=0$, at each branching event we only keep the closest child, and discard all the others. Therefore, at each time $t\geq 0$ there remains at most one particle, called the \emph{selected atom}. With notation of Definition~\ref{Ch4:dfn:BOUPF}, the position of the selected atom is given by
 \begin{equation*}
   Z_*(t):=  \e^{-\dOU (t-b_{\bar{1}_n})}a_{\bar{1}_n} +  Z_{\bar{1}_n}(t-  b_{\bar{1}_n}),\qquad   t\in [ b_{\bar{1}_n}, b_{\bar{1}_n} +\lambda_{\bar{1}_n}), 
 \end{equation*}
 where $\bar{1}_n:= (1,1,\ldots,1)\in \N^n$ for every $n\geq 0$.

 \begin{lem}\label{Ch4:lem:Wselect}
 	 The position of the selected atom $Z_*$
 	is an OU type process with characteristics  $(\Phi_*, \dOU)$, where
 	\begin{equation}\label{Ch4:eq:PhiZ}
 	 \Phi_*(q) = \frac{1}{2} \sigma^2 q^2 + c q 
 	+ \int_{\Rd}  \left(\e^{qr_1}-1 +  q(1-\e^{r_1}) \right) \mu(\dd \rr), \qquad q\geq 0. 
 	\end{equation}
 \end{lem}
The proof of  Lemma~\ref{Ch4:lem:Wselect} is deferred to Appendix~\hyperref[sec:A]{A}. 
Let us now introduce the \emph{cumulant} $\kappa: [0,\infty)\to (-\infty,\infty]$, which will play an important role in this work:   
\begin{equation}\label{Ch4:eq:kappaZ}
     \kappa(q) :=\Phi_*(q)+\int_{\Rd}\sum_{i=2}^{\infty} \e^{q r_i}\mu(\dd \rr)= \frac{1}{2}\sigma^2 q^2 + c q + \int_{\Rd} \left( \sum_{i=1}^{\infty} \e^{q r_i} -1  +q (1-\e^{r_1}) \right) \mu (\dd \rr), \qquad q\ge 0,
\end{equation}  
where $\Phi_*(q)$ is as in \eqref{Ch4:eq:PhiZ}. If \eqref{Ch4:eq:mu} and \eqref{eq:fbi} hold, then $\kappa$ is finite and continuous. 
\begin{lem}\label{Ch4:prop:meanZf}
Let $\Zs$ be an OU type branching Markov process with characteristics $(\sigma, c, \mu, \dOU)$ and suppose that \eqref{eq:fbi} holds.   
Then for every $t\ge 0$ and $q\geq 0$, we have
\begin{equation*}
  \Exp{\sum_{z\in \Zs(t)} \e^{qz} } = \expp{ \int_0^t \kappa(q \e^{-\dOU s}) \dd s }.
\end{equation*}  
\end{lem}

 \begin{proof}
With notation in Definition~\ref{Ch4:dfn:BOUPF}, write $\lambda_{\emptyset}$ for the lifetime of the ancestor and $(a_i:= Z_{\emptyset}(\lambda_{\emptyset})+\Delta a_i, i\in \N)$ for the sequence of positions of the first generation at birth. 
Consider the sub-population generated by the particle $i\in \U$, i.e.
\[\Zs^i (t):= \multiset{\e^{-\dOU (t+\lambda_{\emptyset} -b_{iv})}a_{iv} + Z_{iv}(t+\lambda_{\emptyset}-b_{iv}):~ v\in \U, b_{iv} \leq t + \lambda_{\emptyset}< b_{iv}+ \lambda_{iv}} , \qquad t\geq 0.\]
Conditionally on $(a_i, i\in \N)$, we deduce by \eqref{Ch4:eq:solOU} and Definition~\ref{Ch4:dfn:BOUPF} that the sequence of processes $\big( \Zs^i , i\in \N \big)$ are independent, and each $\Zs^i$ has the same law as the process $(\e^{-\dOU t} a_i + \Zs(t))_{t\geq 0}$. 
Let $m(q, t):= \Exp{\sum_{z\in \Zs(t)} \e^{qz} }$. The decomposition at $\lambda_{\emptyset}$ yields
\begin{eqnarray*} 
   &&  m(q, t) \\
& =& \Prob{\lambda_{\emptyset}>t} \Exp{\e^{ q Z_{\emptyset}(\lambda_{\emptyset})}}  + \Exp{\ind{\lambda_{\emptyset}\leq t}  \sum_{i=1}^{\infty} \expp{ q (Z_{\emptyset}(\lambda_{\emptyset})+\Delta a_i)\e^{-\dOU(t-\lambda_{\emptyset})} } \sum_{z\in \Zs^i(t-\lambda_{\emptyset})} \e^{qz} }\\
& =& \e^{- \mu(\Rdtwo)t} \e^{\int_0^t \psi(q\e^{-\dOU r}) \dd r } +  \int_0^{t}\e^{- \mu(\Rdtwo) s} \e^{\int_{t-s}^t \psi(q\e^{-\dOU r}) \dd r} m(q, t-s) \dd s \int_{\Rdtwo} \sum_{i=1}^{\infty} \e^{q\e^{-\dOU (t-s)} r_i} \mu( \dd \rr). 
\end{eqnarray*}
Changing variable in the integral by $t-s \mapsto s$, we have 
\[   \e^{-\int_0^t \psi(q\e^{-\dOU r}) \dd r } \e^{ \mu(\Rdtwo)t} m(q, t) = 1 + \int_0^{t} \e^{\mu(\Rdtwo) s}  \e^{-\int_0^s \psi(q\e^{-\dOU r}) \dd r } m(q, s)  \Big(\kappa(q \e^{-\dOU s})- \psi(q\e^{-\dOU s})+ \mu(\Rdtwo)\Big) \dd s.\]
Solving this integral equation yields the desired identity. 
 \end{proof}
  
\subsection{OU type branching Markov processes}\label{Ch4:sec:BOUPI}
We next relax the finite birth-intensity assumption \eqref{eq:fbi}, only suppose that \eqref{Ch4:eq:mu} holds, and define OU type branching Markov processes in this more general setting.   Along the lines of \cite[Definition~2]{Bertoin:CF}, our approach relies on the key embedding property, Lemma~\ref{Ch4:lem:cut}. 
 Specifically, for every $\levOU\geq 0$, write $\cutOU{\mu}$ for the image of $\mu$ by the map $\rr \mapsto \cutOU{\rr}$, then we have 
\begin{equation}\label{eq:el}
\cutOU{\mu}(\rr~\colon~ \# \rr > \e^{\levOU})\le \mu(\rr~\colon~ r_{\lceil \e^{\levOU} \rceil} > \e^{\levOU})\le \mu(\rr~\colon~ \sum_{i\ge 1} \e^{r_i} > \lceil \e^{\levOU} \rceil \e^{-\levOU} >1) = 0. 
\end{equation}
Appealing to this fact and \eqref{Ch4:eq:mu}, we hence deduce that \eqref{eq:fbi} holds for $\cutOU{\mu}$: 
\[
\int_{\Rdtwo}\#\rr \cutOU{\mu} (\dd \rr)  \le \lceil \e^{\levOU} \rceil  \mu( \cutOU{\rr} \not\in \Rd_1) 
=\lceil \e^{\levOU} \rceil \mu\big(r_1 = -\infty ~\text{or}~ r_2>-\levOU\big) \leq \lceil \e^{\levOU} \rceil \mu(1-\e^{r_1} > \e^{-\levOU}) <\infty. 
\]
By Lemma \ref{Ch4:lem:cut} and Kolmogorov's extension theorem, we can build a family of processes on the same probability space, which we still denote by $(\cutOU{\Zs})_{\levOU \geq 0}$, such that each $\cutOU{\Zs}$ is an OU type branching Markov process with characteristics $(\sigma, c, \cutOU{\mu}, \dOU)$, and  
\[
\cutOU[\levOU']{(\cutOU{\Zs})} =\cutOU[\levOU']{\Zs} \quad \text{ for every } \levOU' <\levOU, 
\]
where $\cutOU[\levOU']{(\cutOU{\Zs})}$ denotes the process obtained by cutting $\cutOU{\Zs}$ at level $\levOU'$.  

\begin{dfn}\label{Ch4:dfn:BOUP}
Suppose that \eqref{Ch4:eq:mu} holds. In the notation above, we define 
\[
\Zs(t):= \biguplus_{\levOU \in \R} \cutOU{\Zs}(t),\qquad t\geq 0,
\] 
where $\biguplus$ means the union of multisets.  
Then $\Zs$ is called an \emph{OU type branching Markov process} with characteristics $(\sigma, c , \mu, \dOU)$. For every $\levOU\ge 0$, we refer to $\cutOU{\Zs}$ as the truncated process at level $\levOU$. 
\end{dfn}

 \begin{rem}\label{rem:genealogy}
Unlike the finite birth-intensity case, the branching OU type process with infinite birth-intensity does not possess an obvious genealogical structure. Nevertheless, one should be able to build a genealogical structure by using the approach in \cite{SW-tilting} or \cite{BM-blp}. 
 \end{rem}

 The next statement proves that there is no \emph{(local-)explosion}, that is, for every $w\in \R$  and every time $t\ge 0$, only a finite number of the elements of $\Zs(t)$ are larger than $w$. 
Let us still define $\kappa$ as in \eqref{Ch4:eq:kappaZ}, but when $\mu (\Rdtwo)=\infty$, $\kappa$ can possibly take the value $+\infty$; in fact, we observe that, under \eqref{Ch4:eq:mu}, 
\[
  q\in\dom := \left\{ q\geq 0:~\kappa(q)<\infty \right\} \quad \text{~if and only if ~}\quad \int_{\Rd}\sum_{i=2}^{\infty} e^{q r_i} ~\mu(\dd \rr) <\infty.
  \]
 This infers that $\dom$ is a right-unbounded interval. 
  As $\sum_{i=2}^{\infty}e^{q r_i}\leq (1-e^{r_1})^q$ for $q\geq 2$, condition \eqref{Ch4:eq:mu} ensures that $[2, \infty)\subset \dom$. 
Besides, $\kappa$ is continuous and convex in $\dom$.

\begin{thm}\label{Ch4:thm:meanZ}
   Let $\alpha\ge 0$ and suppose that  $\kappa(\alpha)<\infty$. 
Then for every $t\ge 0$ and $q\geq \alpha (1 \vee \e^{\dOU t})$, we have
\begin{equation*}
  \EE \bigg[\sum_{z\in \Zs(t)} \e^{qz} \bigg] = \exp \Big( \int_0^t \kappa(q \e^{-\dOU s}) \dd s \Big). 
\end{equation*}
\end{thm}

 \begin{proof}
 For every $\levOU \geq 0$, recall that the truncated process $\cutOU{\Zs}$ with characteristics $(\sigma, c, \cutOU{\mu}, \dOU)$ has cumulant 
   \begin{equation*}
     \cutOU{\kappa}(q) = \frac{1}{2}\sigma^2q^2 + c q + \int_{\Rd} \left(\e^{q r_1} + \sum_{i=2}^{\infty}\ind{r_i > -\levOU} \e^{q r_i} -1  +q (1-\e^{r_1}) \right) \mu (\dd \rr), \qquad q\geq 0.
   \end{equation*}
Since $\mu$ fulfills \eqref{Ch4:eq:mu}, both \eqref{Ch4:eq:mu} and \eqref{eq:fbi} hold for $\cutOU{\mu}$. Then Lemma~\ref{Ch4:prop:meanZf} yields
\[ \EE \bigg[ \sum_{z\in \cutOU{\Zs}(t)} \e^{qz} \bigg]= \exp \Big( \int_0^t \cutOU{\kappa}(\e^{-\dOU s} q) \dd s \Big), \qquad q\geq 0. \]
Letting $\levOU\to \infty$, it is plain that for every $p\geq \alpha$, there is $\kappa(p)<\infty$ and 
$\lim_{\levOU \to \infty}\uparrow \cutOU{\kappa}(p) = \kappa(p)$. 
We hence deduce the claim by monotone convergence.
 \end{proof}

 \begin{rem}\label{rem:general}
The OU type branching Markov processes built here are tailored to make connections with growth-fragmentations. The same construction would be equally applicable for a more general setting: an atom could move according to any Markov process (with possibly positive jumps), and $\mu$ could be a measure on the space of all point measures on $\R$. The crucial point is to find a non-explosion condition, which should be a proper intergrable condition similar to \eqref{Ch4:eq:mu}; see also \cite[Equation (1.3)]{BM-blp} and \cite[Theorem~2]{Bertoin:GF-Markovian} for non-explosion conditions in various circumstances.  
 \end{rem}

\section{OU type growth-fragmentation processes}\label{Ch4:sec:OUGFP}

\subsection{The model and its basic properties}\label{Ch4:sec:construction}
We are now ready to define OU type growth-fragmentation processes. 
Let $\sigma\geq 0$, $c\in \R$, $\dOU\in \R$ and $\nu$ be a sigma-finite measure on the space of mass-partitions $ \Sd$, fulfilling \eqref{Ch4:eq:nu}.   
Write $\mu$ for the image of $\nu$ by the map $ (s_1, s_2, \ldots) \mapsto (\log s_1, \log s_2, \ldots)\in \Rd$, then $\mu$ is a sigma-finite measure on $\Rd$, and \eqref{Ch4:eq:nu} ensures that $\mu$ satisfies \eqref{Ch4:eq:mu}. Hence we are allowed to construct by Definition~\ref{Ch4:dfn:BOUP} an OU type branching Markov process $\Zs$ with characteristics $(\sigma, c , \mu, \dOU)$.
Recall that $\co$ is the space of all decreasing null sequences endowed with the $\ell^{\infty}$-distance, i.e. $\|\xs - \mathbf{y} \|_{\infty} = \sup_{i\in \N}|x_i -y_i|$ for $\xs= (x_1, x_2, \ldots)\in \co$ and $\mathbf{y}= (y_1, y_2, \ldots)\in \co$. Theorem~\ref{Ch4:thm:meanZ} enables us to give the following definition. 

\begin{dfn}\label{Ch4:dfn:OUGF}
For every $t\geq 0$, the elements of $\multiset{\exp(z):~ z\in \Zs(t) }$ can be rearranged in a decreasing null sequence 
$$\Xsd(t):= (X_1(t), X_2(t), \ldots)\in \co.$$
The process $\Xsd:=(\Xsd(t), t\geq 0)$ is called an \emph{OU type growth-fragmentation process} with characteristics $(\sigma, c , \nu, \dOU)$.
\end{dfn}

Roughly speaking, $\sigma\geq 0$ describes the fluctuations of the size, the constant $c\in \R$ represents the deterministic dilation (resp. erosion) coefficient when $c>0$ (resp. $c<0$). The measure $\nu$ is called the \emph{dislocation measure}. For every $\sd\in \Sd$, a fragment of size $x>0$ splits into a sequence of fragments $x \sd$ at rate $\nu(\dd \sd)$. The constant $\dOU\in \R$ characterizes the speed at which the size of a fragment evolves towards (when $\dOU >0$) or away from (when $\dOU <0$) the value $1$ (as the central location of an OU type process is $0$).

\begin{rem}
In the following, an OU type growth-fragmentation $\Xsd$ is always assumed (without loss of generality) to start from one fragment of unit size, i.e. $\Xsd(0):= (1,0,0,\ldots)$, unless otherwise specified.   
\end{rem}
\begin{rem}
  When $\dOU=0$, an OU type growth-fragmentation with characteristics $(\sigma, c , \nu, 0)$ is a compensated fragmentation with characteristics $(\sigma, c , \nu)$ in the sense of \cite[Definition~3]{Bertoin:CF}. To avoid duplication, this case will be implicitly excluded hereafter. 
\end{rem}

Theorem~\ref{Ch4:thm:meanZ} can be easily transferred to OU type growth-fragmentations. 
  Correspondingly, the \emph{cumulant} of $\Xsd$ is given by
  \begin{equation}\label{Ch4:eq:kappaOU}
    \kappa(q) := \frac{1}{2}\sigma^2 q^2 + c q + \int_{\Sd} \bigg( \sum_{i=1}^{\infty} s_i^q -1  +q (1-s_1) \bigg) \nu (\dd \sd), \qquad q\geq 0. 
  \end{equation}
By the discussion before Theorem~\ref{Ch4:thm:meanZ}, as \eqref{Ch4:eq:nu} holds, we still have that $\dom$ is a right-unbounded interval containing $[2,\infty)$, and that $\kappa$ is convex and continuous on $\dom$. 

\begin{thm}\label{Ch4:thm:mean}
For every $t\geq 0$, $\alpha\in \dom$ and $q\geq \alpha (1 \vee \e^{\dOU t})$, we have
\begin{equation*}
  \Exp{\sum_{i=1}^{\infty} X_i(t)^q} = \expp{ \int_0^t \kappa(q \e^{-\dOU s}) \dd s }. 
\end{equation*}
\end{thm}
\begin{proof}
It follows directly from Theorem~\ref{Ch4:thm:meanZ}. 
\end{proof}

For every $\levOU\ge 0$, we refer to the exponential of the truncated process $\cutOU{\Zs}$ (rearranged in decreasing order) as the \emph{truncated} OU type growth-fragmentation $\cutOUd{\Xs}$.  
When $\levOU =0$, in the truncated system $\cutOUd[0]{\Xs}$ there is at most one fragment at any time, called \emph{the selected fragment} of $\Xsd$; however, it is not necessarily the largest one in the system.

\begin{lem}[Selected fragment]\label{Ch4:lem:selected}
The size of the selected fragment $(X_*(t),~t\geq 0) $ is the exponential of an OU type process with characteristics $(\Phi_*, \dOU)$, where 
\begin{equation}\label{Ch4:eq:Phi-selected}
\Phi_*(q) := \frac{1}{2} \sigma^2 q^2 + c q 
+ \int_{\Sd}  \big(s_1^q-1 +  q(1-s_1) \big) \nu(\dd \sd), \qquad q\geq 0.
 \end{equation}  
\end{lem}
\begin{proof}
	The law of $\log X_*$ is given by Lemma~\ref{Ch4:lem:Wselect}.
\end{proof}

With the help of Theorem~\ref{Ch4:thm:mean}, we shall establish some fundamental properties of $\Xsd$ in the rest of this subsection. 
We first prove that $\Xsd$ is a time-homogeneous Markov process. In this direction, let us define a family of probability measures. 
Specifically, let $\alpha\in \dom$ and $\xs= (x_1, x_2, \ldots) \in \lp[\alpha]\subset \co$, 
where $\lp[\alpha]$ denotes the space of decreasing null sequences with finite $\ell^{\alpha}$-norm, i.e. $\|\xs \|_{\ell^{\alpha}} := (\sum_{i=1}^{\infty} |x_i|^{\alpha} )^{\frac{1}{\alpha}}<\infty$. Let $(\Xs^{[j]}, j\in \N)$ be a sequence of i.i.d. copies of $\Xsd$. 
We have for every $t\geq 0$ and $q\geq \alpha(\e^{\dOU t}\vee 1)$ that 
\begin{equation*}
  \Exp{\sum_{j\geq 1} \sum_{i\geq 1} \left| x_{j}^{\e^{-\dOU t}}  X^{[j]}_i(t) \right|^q }
 = \expp{-\int_0^t \kappa(q \e^{-\dOU s}) \dd s }  \sum_{j\geq 1} |x_j|^{q\e^{-\dOU t}} <\infty,
\end{equation*}
so the elements (repeated according to their multiplicity) of $\{ x_j^{\e^{-\dOU t}} X^{[j]}_i(t),~ i, j\in \N \}$ can be rearranged in decreasing order. Write $\Ps_{\xs}$ for the law of the resulting process on $\co$. 
\begin{prop}[Markov property]\label{Ch4:prop:branching}
For every $s\ge 0$, conditionally on $(\Xsd(r), 0\leq r\leq s)$, the process $(\Xsd(t+s), t\geq 0)$ on $\co$ has distribution $\Ps_{\Xsd(s)}$.
\end{prop}
This statement clearly ensures that $\Xsd$ fulfills the properties \ref{Ch4:P1} and \ref{Ch4:P2} in the introduction.

\begin{proof}
We first derive from Theorem~\ref{Ch4:thm:mean} that $\Xsd(s)\in \lp[2 (1 \vee \e^{\dOU s})]$. 
Since $2 (1 \vee \e^{\dOU s})\in \dom$ always holds, the law $\Ps_{\Xsd(s)}$ is indeed well-defined. 

For every $\levOU \geq 0$, consider the \emph{truncated} OU type growth-fragmentation $\cutOUd{\Xs}$. 
It is plain from Definition~\ref{Ch4:dfn:BOUPF} that $\cutOUd{\Xs}$ fulfills the claimed Markov property. This observation and Theorem~\ref{Ch4:thm:mean} entail that the Markov property also holds for $\Xsd$. 
See \cite[proof of Proposition~2]{Bertoin:GF-Markovian} for similar arguments and we omit the details.  
\end{proof}

\begin{rem}
It would be interesting to characterize all c\`adl\`ag $\co$-valued processes that possess properties \ref{Ch4:P1} and \ref{Ch4:P2}. 
It is intuitive to guess that all such processes are the exponential of OU type branching Markov processes (with possibly positive jumps and general offspring distribution as discussed in Remark~\ref{rem:general}). 
 For the homogeneous case $\theta=0$, this has been confirmed by \cite{BM-blp}.  
 However, the extension of this result to the general case is beyond the scope of this paper and will be taken up separately. 
\end{rem}
We next obtain the following non-negative martingales, which should be compared with the well-known \emph{additive martingales} in the context of (pure) fragmentations \cite{BertoinRouault} or branching random walks \cite{Biggins:1977}.
\begin{prop}[Additive martingales]\label{Ch4:prop:mart} 
Let $\Xsd$ be an OU type growth-fragmentation with cumulant $\kappa$ and starting point $\Xsd(0)=(x,0,0,\ldots)$. 
  \begin{enumerate}[label=(\roman*)]
  \item If $\dOU<0$, then for every $q\in \dom$, the process
\[\bigg( x^{-q \e^{-\theta t}} \exp \Big(-\int_0^t \kappa(q \e^{-\dOU s}) \dd s \Big)\sum_{i=1}^{\infty} X_i(t)^q , \quad t\geq 0 \bigg)\quad  \text{is a martingale}.\]

\item   If $\dOU > 0$, then for every $\alpha \in \dom$, the process
\[ \bigg(  x^{-\alpha}\exp \Big( -\int_0^t \kappa(\alpha \e^{\dOU s}) \dd s \Big)\sum_{i=1}^{\infty} X_i(t)^{\alpha \e^{\dOU t}} , \quad t\geq 0\bigg)\quad  \text{is a martingale}.\]
  \end{enumerate}
\end{prop}
\begin{proof}
We deduce from Theorem~\ref{Ch4:thm:mean} and Proposition~\ref{Ch4:prop:branching} that both processes have a constant mean value, which is $1$. Then the martingale property follows from Proposition~\ref{Ch4:prop:branching}. 
\end{proof}

\begin{rem}\label{rem:spine}
The (non-negative) additive martingale induces a natural change of measures. Using the methods developed in \cite{SW-tilting}, this would enable us to develop the spinal techniques introduced in the seminal work \cite{Spine}, which are important tools in the study of branching processes. A potential application is to determine whether the limit of the additive martingale is degenerate, for example. 
Note that to define a spine, we also need a genealogical structure (see Remark~\ref{rem:genealogy}). 
\end{rem}

\begin{prop}[Feller-type property] \label{Ch4:prop:Feller} 
Let $\alpha \in \dom$ and suppose that a sequence $\xs_n \to \xs_{\infty}$ in $\lp[\alpha]$. Then for every $t\geq 0$, there is the weak convergence
\[ ( \Ps_{\xs_n}(s), s\in [0,t]) \underset{n\to \infty}{\Longrightarrow} (\Ps_{\xs_{\infty}}(s), s\in [0,t])\]
 in the sense of finite dimensional distributions on $\lp[q]$ for every $q\geq \max( \alpha(\e^{\dOU t}\vee 1) ,1)$.
\end{prop}

\begin{proof}
The idea is from the proof of \cite[Corollary~2]{Bertoin:CF}, but different estimations are needed for our case. Consider a sequence $(\Xs^{[j]}, j\in \N)$ of i.i.d. copies of $\Xs$. As $q\geq \alpha(\e^{\dOU t}\vee 1)$. It follows from Theorem~\ref{Ch4:thm:mean} that 
 \begin{equation}\label{Ch4:eq:feller:proof}
   \Exp{\sum_{j=1}^{\infty}\sum_{i=1}^{\infty} \left| (x_{n,j}^{\e^{-\dOU t}}-x_{\infty,j}^{\e^{-\dOU t}})  X^{[j]}_i(t) \right|^q }
 = \expp{-\int_0^t \kappa(q \e^{-\dOU s}) \dd s }  \sum_{j=1}^{\infty} |x_{n,j}^{\e^{-\dOU t}}-x_{\infty,j}^{\e^{-\dOU t}}|^q .
 \end{equation}
If $\dOU>0$, as the function $x\mapsto x^{\e^{-\dOU t}}$ is concave, then for every $j\geq 1$ there is
\[
|x_{n,j}^{\e^{-\dOU t}}-x_{\infty,j}^{\e^{-\dOU t}}| \leq |x_{n,j}-x_{\infty,j}|^{\e^{-\dOU t}}.
\]
We next consider the case $\dOU<0$. Since $\xs_n \to \xs_{\infty}$ in $\lp[\alpha]$, we may assume that for every $n\geq 1$, there is $|x_{n,j}-x_{\infty,j}|<1$ for every $j\geq 1$, so $ \|\xs_n \|_{\ell^{\infty}} \leq  \|\xs_{\infty} \|_{\ell^{\infty}}+1$. 
Therefore, with a constant $C(t):=\e^{-\dOU t}(\|\xs_{\infty} \|_{\ell^{\infty}}+1)^{\e^{-\dOU t} -1}$, we have
\[
|x_{n,j}^{\e^{-\dOU t}}-x_{\infty,j}^{\e^{-\dOU t}}| \leq C(t)|x_{n,j}-x_{\infty,j}|,\qquad ~\text{for every}~j\in\N. 
\]
Combining these observations and that $\xs_n \to \xs_{\infty}$ in $\lp[\alpha]$, we deduce from \eqref{Ch4:eq:feller:proof} that 
\[
\lim_{n\to \infty}\Exp{\sum_{j= 1}^{\infty} \sum_{i= 1}^{\infty} \left| (x_{n,j}^{\e^{-\dOU t}}-x_{\infty,j}^{\e^{-\dOU t}})  X^{[j]}_i(t) \right|^q } = 0. 
\]
Write $\mathbf{x}^{\downarrow}$ and $\mathbf{y}^{\downarrow}$ for the decreasing rearrangements of two sequences $\mathbf{x}$ and $\mathbf{y}$ in $\ell^q$. As the function $x\mapsto x^q$ is convex for $q\geq 1$, it follows from \cite[Theorem~3.5]{LiebLoss:Analysis} that $\|\mathbf{x}^{\downarrow}-\mathbf{y}^{\downarrow}\|^q_{\ell^q}\leq \|\mathbf{x}- \mathbf{y}\|^q_{\ell^q}$. As a consequence, there is 
\[
 \left\|\big(x_{n,j}^{\e^{-\dOU t}} X^{[j]}_i(t)\big)^{\downarrow}- \big(x_{\infty,j}^{\e^{-\dOU t}}  X^{[j]}_i(t)\big)^{\downarrow} \right\|^q_{\ell^q}
 \leq \sum_{j=1}^{\infty} \sum_{i=1}^{\infty} \left| (x_{n,j}^{\e^{-\dOU t}}-x_{\infty,j}^{\e^{-\dOU t}})  X^{[j]}_i(t) \right|^q ,
\]
which leads to
\[
 \lim_{n\to \infty}\Exp{\left\|\big(x_{n,j}^{\e^{-\dOU t}} X^{[j]}_i(t)\big)^{\downarrow}- \big(x_{\infty,j}^{\e^{-\dOU t}}  X^{[j]}_i(t)\big)^{\downarrow} \right\|_{\ell^q} } = 0. 
\]
From the description of $\Ps_{\xs_n}$ and $\Ps_{\xs_{\infty}}$, we deduce the Feller-type property. 
\end{proof}

We finally establish the sample path regularity. 
 \begin{prop}[C\`adl\`ag path]\label{Ch4:prop:cadlag}
Let $\alpha \in \dom$, $T\geq 0$ and $q\geq \max(\alpha (\e^{\dOU T}\vee 1),1)$. Then the process $(\Xsd(t),t\in [0,T])$ possesses a c\`adl\`ag version in $\lp[q]$. Thus, the non-stopped process $(\Xsd(t),t\in [0,\infty)$ possesses a c\`adl\`ag version in $\co$.  
\end{prop}

 \begin{proof}
We follow the same arguments as in \cite[proof of Proposition~2]{Bertoin:CF}. For every $\levOU\geq 0$, let $\cutOU{\Zs}$ be the truncated OU type branching Markov process and $\cutOU{\Xs}$ be its associated growth-fragmentation, then it follows plainly from the construction that $(\cutOU{\Xs}(t), t\in[0,T])$ is almost surely c\`adl\`ag in $\lp[q]$. Therefore, to complete the proof, it suffices to prove that
\begin{equation}\label{Ch4:eq:XltoX}
   \lim_{\levOU \to \infty} \sup_{0\leq t\leq T} \| \Xs(t) -\cutOU{\Xs}(t) \|_{\ell^q}^{q} = 0 \quad \text{in probability}.
\end{equation}

Recall that the operation of rearranging two sequences of positive numbers in the decreasing order decreases their $\ell^q$-distance, we can easily deduce the following inequality (see \cite[Lemma~4]{Bertoin:CF} for details): 
\begin{equation}\label{Ch4:eq:Xl-X}
\| \Xsd(t) -\cutOU{\Xsd}(t) \|_{\ell^q}^q \leq \| \Xsd(t)\|_{\ell^q}^q -\|\cutOU{\Xsd}(t) \|_{\ell^q}^q, \quad \text{for every } t\in [0,T].
\end{equation}
By this inequality and the fact that $\kappa\geq \cutOU{\kappa}$, we deduce that
\[
\sup_{0\leq t\leq T} \| \Xsd(t) -\cutOU{\Xsd}(t) \|_{\ell^q}^q \leq A \sup_{0\leq t\leq T} \Big| M(t) -\cutOU{M}(t)  \Big|+ B(\levOU) \sup_{0\leq t\leq T}M(t),
\]
where $M(t):= \exp\big(-\int_0^t \kappa(q \e^{-
	 \dOU r}) \dd r\big)\| \Xsd(t)\|_{\ell^q}^q$, $\cutOU{M}(t):= \exp\big(-\int_0^t \cutOU{\kappa}(q \e^{-\dOU r}) \dd r\big)\| \cutOU{\Xsd}(t)\|_{\ell^q}^q$, $A := \sup_{0\leq t\leq T} \exp\Big(\int_0^t \kappa(q \e^{-\dOU r}) \dd r\Big)$ is a finite constant, and 
	\[B(\levOU):=  \sup_{0\leq t\leq T}\bigg( \exp\Big(\int_0^t \kappa(q \e^{-\dOU r}) \dd r\Big) - \exp\Big(\int_0^t \cutOU{\kappa}(q \e^{-\dOU r}) \dd r\Big)\bigg) \underset{\levOU\to \infty}{\longrightarrow}  0.
	\]
	We know by monotone convergence that $\lim_{\levOU\to \infty}\uparrow\| \cutOU{\Xsd}(T)\|_{\ell^q}^q=\| \Xsd(T)\|_{\ell^q}^q$. Since $q\geq \alpha (\e^{\dOU T}\vee 1)$, it follows from Theorem~\ref{Ch4:thm:mean} that $\Exp{\|\Xsd(T)\|_{\ell^q}^q}<\infty$. Then by dominated convergence we have
	  \[\lim_{\levOU\to \infty} \Exp{\left|M(T)-\cutOU{M}(T)\right|}=0.\]
	   	Since it follows from Proposition~\ref{Ch4:prop:mart} that $(M(t), t\in [0,T])$ and $(\cutOU{M}(t), t\in [0,T])$ are both martingales, using Doob's inequality leads to \eqref{Ch4:eq:XltoX}. 
           We have completed the proof. 
 \end{proof}

\begin{rem}\label{Ch4:rem:SMP}
As a consequence of the Feller-type property and the c\`adl\`ag path, we deduce that $\Xsd$ fulfills the strong Markov property by a standard argument (approximate a general stopping time by a decreasing sequence of simple stopping times, and the Markov property holds for simple stopping times).
\end{rem}

\subsection{Many-to-one formula and growth-fragmentation equations}\label{Ch4:sec:mto}

Let $x>0$ and $\Xsd := (\Xsd(t) = (X_1(t), X_2(t),\ldots ),~ t\geq 0)$ be an OU type growth-fragmentation process on $\co$ with characteristics $(\sigma, c, \nu, \theta)$, starting from $\Xsd(0)= (x,0,0,\ldots)$. 
For every $t\geq 0$, define a measure $\rho_x (t)(\dd y)$ on $\R_+= (0,\infty)$, such that for every $f \in \Cinfty$ (the space of $C^{\infty}$-functions on $\R_+$ with compact support), the identity holds: 
\begin{equation}\label{Ch4:eq:rho}
\prm{\rho_{x}(t)}{ f}:=\int_{\R_+} f(y)\rho_{x}(t) (\dd y) = \Exp[x]{\sum_{i=1}^{\infty} f(X_i(t))  }.
\end{equation}
Informally speaking, $\rho_x$ is the ``mean value'' of $\Xsd$. 

In this subsection we study the evolution of $\rho_x(t)$ as time proceeds. We first aim at expressing $\rho_{x}$ in term of the transition kernel of a certain Markov process. 
This idea is often referred to as a \emph{many-to-one formula} in the literature, and it has been widely used in the study of branching type processes; see e.g. \cite{Shi:BRW-Book} for branching random walks  and \cite[Theorem~3.5]{BCK:Martingale} for self-similar growth-fragmentations. 
We shall treat the inward ($\theta>0$) and outward ($\theta<0$) cases separately. For the inward case, we need a certain time-inhomogeneous \emph{affine Markov process} $\chi$. We refer to \cite{DFS:ap, Fil:tiap} for a general study of the latter. 
To describe the process $\chi$, let us record the following observation, which extends \cite[Lemma~3.1]{BertoinWatson}.
\begin{lem}\label{Ch4:lem:Htransform}
For every $\alpha \in \dom$, there exists a spectrally negative L\'evy process $\xi_{\alpha}$ with Laplace exponent 
\[\Phi_{\alpha}(q) := \kappa(q+ \alpha ) -\kappa(\alpha) ,\qquad q\geq 0.\]
Specifically, the L\'evy process $\xi_{\alpha}$ has characteristics
$(\sigma, c_{\alpha}, \Lambda_{\alpha}, 0)$, where 
\[c_{\alpha} := c+ \sigma^2 \alpha + \int_{\Sd} \Big( (1-s_1)- \sum_{i=1}^{\infty} s_i^{\alpha}(1-s_i)  \Big) \nu(\dd \sd) ,\]
 and the L\'evy measure $\Lambda_{\alpha}$ on $(-\infty, 0)$ is defined such that for every bounded measurable function $g$ on $(-\infty, 0)$ there is the identity
\[\int_{(-\infty, 0)} g(z) \Lambda_{\alpha}(\dd z) = \int_{\Sd} \sum_{i=1}^{\infty} \ind{s_i>0} s_i^{\alpha} g(\log s_i) \nu(\dd \sd).\]
\end{lem}
We omit the proof, which is straightforward.  

 \begin{lem}\label{lem:affine}
 Suppose that $\theta>0$. Then for every $\alpha\in \dom$, there exists a unique time-inhomogeneous affine process $\chi:=\chi^{(\alpha)}$ (that depends on $\alpha$) with state space $\R$, whose transition kernel $(P^{\chi}_{t,T}(z,\dd w))_{0\le t\le T}$ is determined by the Laplace transform 
\begin{equation*}
 \int_{\R} \e^{q w} P^{\chi}_{t,T}(z,\dd w)   = \exp \Big( \psi(t,T,q) z + \phi(t,T,q)  \Big), \qquad q\ge 0,
\end{equation*}
where 
$\psi(t,T,q):= \e^{-\theta(T-t) }  q$ and $\phi(t,T,q):=\int_t^{T} \Big( \kappa(q \e^{-\theta(T-r)} +\alpha \e^{\theta r}) - \kappa(\alpha \e^{\theta r}) \Big) \dd r.$ 
The associated time-homogeneous process $(t,\chi(t))_{t\ge 0}$ is a Feller process with infinitesimal generator $\mathcal{A}$. Furthermore, every $C^{1,2}$-function $g$ on $\R_+\times \R$ with compact support belongs to the domain of $\mathcal{A}$, and we have
\begin{align*}
\mathcal{A} g(t,z):=&\partial_t g(t,z) +\frac{1}{2} \sigma^2 \partial^2_{zz} g(t,z) + (c_{\alpha \e^{\theta t}} -\theta z) \partial_z g(t,z) \\
& \quad +\int_{(-\infty, 0)} \Big( g(t, z+w) - g(t, z) + (1-\e^w) \partial_z g(t, z) \Big)  \Lambda_{\alpha \e^{\theta t}}(\dd w),
\end{align*}
where $c_{\alpha \e^{\theta t}}$ and $\Lambda_{\alpha \e^{\theta t}}$ are as in Lemma~\ref{Ch4:lem:Htransform}. 
  \end{lem}
  \begin{proof}
 With notation of \cite[Definition~2.5]{Fil:tiap}, let $a_{\chi}(t)= \frac{1}{2} \sigma^2$, $\alpha_{\chi}(t)=0$, $b_{\chi}(t)=c_{\alpha \e^{\theta t}}$, $\beta_{\chi}(t)= -\theta$, $c_{\chi}(t)=0$, $\gamma_{\chi}(t)=0$,  $\mu_{\chi}(t,\dd w)$ be a null measure, and $m_{\chi}(t,\dd w)= \Lambda_{\alpha \e^{\theta t}} (\dd w)$.\footnote{For consistency, here we still use $w\mapsto 1-\e^{w}$ as the truncation function; though it is not the same as the one used in \cite{Fil:tiap}, all the results therein still hold, up to a modification of the drift coefficient $b_{\chi}$.}
We can easily check that these parameters are \emph{strongly admissible} in the sense of \cite[Definition~2.5]{Fil:tiap}, and obtain the following functions define by \cite[Equations~(2.16)--(2.18)]{Fil:tiap}:
 \[
R_{\chi}(t,q)= -\theta q, \quad \text{and} \quad F_{\chi} (t,q) = \kappa(q +\alpha \e^{\theta t}) - \kappa(\alpha \e^{\theta t}) , \qquad q\ge 0,
 \]
 where we have used Lemma~\ref{Ch4:lem:Htransform} to get $F_{\chi}$. 
 Then it follows from \cite[Theorem~2.13]{Fil:tiap} that there exists a unique \emph{strongly regular affine Markov process} $\chi$ associated with these parameters, and that $\chi$ has the  transition kernel $P^{\chi}$ and the infinitesimal generator $\mathcal{A}$ as in the statement. We complete the proof. 
\end{proof}

\begin{prop}[Many-to-one formula for the inward case]\label{Ch4:prop:mtoi}
Suppose that $\theta>0$. 
Let $\alpha\in \dom$, $\chi=\chi^{(\alpha)}$ be the  time-inhomogeneous affine process given as in Lemma~\ref{lem:affine}, and $(P^{\e^{\chi}}_{t,T})_{0\le t\le T}$ be the transition kernel of the process $\e^{\chi}:=(\e^{\chi(t)},t\ge 0)$. 
Then for every $x >0$ and $t\ge 0$, there is the identity 
\begin{equation*}
   \rho_{x}(t)(\dd y)  \eqdis x^{\alpha} y^{- \alpha\e^{\theta t}} \e^{\int_0^t \kappa(\alpha \e^{\theta r}) \dd r } P^{\e^{\chi}}_{0,t}(x,\dd y) ,\qquad y>0.
 \end{equation*}
\end{prop}
\begin{proof}
Define a measure by $\tilde{\rho}_{x}(t)(\dd y)  := x^{-\alpha} y^{ \alpha\e^{\theta t}} \e^{-\int_0^t \kappa(\alpha \e^{\theta r}) \dd r }\rho_{x}(t)(\dd y)$, then we deduce from Theorem~\ref{Ch4:thm:mean} that 
\[
\int_{\R_+} y^q \tilde{\rho}_{x}(t)(\dd y) = x^{q \e^{-\theta t} } \exp \Big( \int_0^{t} \big( \kappa(q \e^{-\theta(t-r)} +\alpha \e^{\theta r}) - \kappa(\alpha \e^{\theta r}) \big) \dd r\Big), \qquad q\ge 0.
\]
By Lemma~\ref{lem:affine}, we know that $\tilde{\rho}_{x}(t)(\dd y)$ and $P^{\e^{\chi}}_{0,t}(x,\dd y)$ are the same probability measure. Then  the claim follows. 
\end{proof}

For the outward case with $\theta<0$, since we cannot guarantee $\kappa(\alpha \e^{\theta r})$ to be finite for all $r \ge 0$, the process $\chi$ as in Lemma~\ref{lem:affine} is not well-defined in general. 
However, for any finite time period we can obtain a similar many-to-one formula. 

\begin{prop}[Many-to-one formula for the outward case]\label{Ch4:prop:mtoo}
Let $T_0>0$ and $\alpha\ge 0$. Suppose that $\theta<0$ and $\alpha \e^{\theta T_0} \in \dom$. Then there exists a unique time-inhomogeneous affine Markov process $\chi:= \chi^{(\alpha)}$ on $\R_+$ with infinitesimal generator $\mathcal{A}$: for every $C^{1,2}$-function $g$ on $\R_+\times \R$ with compact support,
\begin{align*}
\mathcal{A} g(t,z):=&\partial_t g(t,z) +\frac{1}{2} \sigma^2 \partial^2_{zz} g(t,z) + (c_{\alpha \e^{\theta (t\wedge T_0)}} -\theta z) \partial_z g(t,z) \\
& \quad +\int_{(-\infty, 0)} \Big( g(t, z+w) - g(t, z) + (1-\e^w) \partial_z g(t, z) \Big)  \Lambda_{\alpha \e^{\theta (t\wedge T_0)}}(\dd w). 
\end{align*} 
The transition kernel of its exponential $\e^{\chi}$ is given by
\begin{equation*}
 \int_{\R} y^q P^{\e^{\chi}}_{t,T}(x,\dd y)   =  x^{q\e^{-\theta (T-t)}} \exp\Big( \int_t^{T} \big( \kappa(q \e^{-\theta(T- r)} +\alpha \e^{\theta (r\wedge T_0)}) - \kappa(\alpha \e^{\theta (r\wedge T_0)}) \big) \dd r \Big), \qquad q\ge 0.
\end{equation*}
Besides, for every $x>0$ and $t\in [0,T_0]$, there is the identity
\begin{equation*}
   \rho_{x}(t)(\dd y)  \eqdis x^{\alpha} y^{- \alpha\e^{\theta t}} \e^{\int_0^t \kappa(\alpha \e^{\theta r}) \dd r } P^{\e^{\chi}}_{0,t}(x,\dd y) ,\qquad y>0.
 \end{equation*}
\end{prop}

\begin{proof}
Recall that $\dom$ includes a right-unbounded interval, then $\alpha \e^{\theta T_0} \in \dom$ infers that $[\alpha \e^{\theta T_0}, \alpha]\subset \dom$. Therefore, with notation of \cite[Definition~2.5]{Fil:tiap}, we can define $a_{\chi}(t)= \frac{1}{2} \sigma^2$, $\alpha_{\chi}(t)=0$, $b_{\chi}(t)=c_{\alpha \e^{\theta (t\wedge T_0)}}$, $\beta_{\chi}(t)= -\theta$, $c_{\chi}(t)=0$, $\gamma_{\chi}(t)=0$,  $\mu_{\chi}(t,\dd w)$ to be a null measure, and $m_{\chi}(t,\dd w)= \Lambda_{\alpha \e^{\theta (t\wedge T_0)}} (\dd w)$. In other words, these are defined in the same way as in the inward case for $t\in[0,T_0]$, and moreover extended to $[0,\infty)$ by a simple continuous extension. 
Due to the continuity, it follows again from \cite[Theorem~2.13]{Fil:tiap} that, there exists a unique strongly regular affine Markov process $\chi$ associated with these parameters, such that $\chi$ has the desired generator and $\e^{\chi}$ has the desired transition kernel. Then we easily derive the many-to-one formula by the same arguments as in the inward case. 
\end{proof}

\begin{rem}\label{rem:mto0}
If $0 \in \dom$, then the affine process $\chi^{(0)}$ in the many-to-one formula is just an OU type process with characteristics $(\Phi_0, \theta)$, where $\Phi_0(q):= \kappa(q)-\kappa(0)$. 
\end{rem}
Using the many-to-one formula, we next describe $\rho_x$  by a \emph{growth-fragmentation equation}. See \cite{BertoinWatson} and  \cite[Corollary~3.12]{BCK:Martingale} for analogous results for self-similar growth-fragmentations.  

\begin{prop}[Growth-fragmentation equation]\label{Ch4:prop:gf-eq}
For every $x>0$, the family of Radon measures $(\rho_{x}(t), t\geq 0)$, given by \eqref{Ch4:eq:rho}, is the unique solution to the growth-fragmentation equation 
\begin{equation}\label{Ch4:eq:gf}
 \prm{ \rho_{x}(t)}{ f} =f(x)+ \int_0^t \prm{\rho_{x}(r)}{\Lc f} \dd r, \qquad \forall f\in \Cinfty,
\end{equation} 
where 
\begin{equation}\label{Ch4:eq:Lc}
\Lc f(y) := \frac{1}{2}\sigma^2 y^{2} f''(y) +\Big( c+\frac{1}{2}\sigma^2- \dOU \log y\Big)y f'(y) + \int_{\Sd} \Big(\sum_{i=1}^{\infty} f(y  s_i) - f(y) + y f'(y)(1-s_1) \Big) \nu(\dd \sd).
\end{equation}
\end{prop}

\begin{proof}
We only prove it for the outward case $\theta<0$; for the inward case the arguments are very similar. 
Throughout the proof, let $T_0\ge 0$ and $\alpha:= \alpha'\e^{-\theta T_0}$ with  $\alpha'\in \dom$. Let $\chi= \chi^{(\alpha)}$ be as in Proposition~\ref{Ch4:prop:mtoo}, with transition kernel $(P^{\chi}_{t,T})_{t\le T}$ and generator $\mathcal{A}$. 

Let us first prove the uniqueness of the solution. 
For every $z\in \R$, suppose that $\rho'_{\e^z}$ is a solution to \eqref{Ch4:eq:gf}. 
For every $0\le t\le T$, define a measure $P'_{t,T} (z,\dd w)$ on $\R$ as follows: 
if $0\le t\le T\le T_0$, then $P'_{t,T} (z,\dd w)$ is the image via $y\mapsto \log y$ of the measure 
\[
\e^{-\alpha \e^{\theta t} z} y^{\alpha \e^{\theta T}} \exp \Big(-\int_t^T \kappa(\alpha \e^{\theta r}) \dd r \Big)   \rho'_{\e^z} (T-t)(\dd y);
\] 
otherwise, we simply write $P'_{t,T}:= P^{\chi}_{t,T}$. 
Then, for every $s,t \ge 0$ and $C^{1,2}$-function $g$ on $\R_+\times \R$ with compact support, we can check  the identity (see Appendix~\ref{sec:B} for details)
\begin{equation}\label{eq:gfe-unique}
 \prm{ P'_{s,s+t}(z,\cdot)}{ g(s+t,\cdot)}  =g(s,z)+ \int_0^t \prm{P'_{s,s+r}(z,\cdot)}{\mathcal{A} g(s+r,\cdot)} \dd r. 
\end{equation}
By the uniqueness of the affine process $\chi$ \cite[Theorem~2.13]{Fil:tiap}, we identify $P'_{t,T}(z,\cdot)= P^{\chi}_{t,T}(z,\cdot)$. This infers that, at any time $t\in [0,T_0]$, all the solutions to \eqref{Ch4:eq:gf} have the same value. 
The arbitrariness of $T_0$ leads to the uniqueness for all time. 

We now check that $\rho_{x}$ is a solution to \eqref{Ch4:eq:gf}. 
For every function $f\in \Cinfty$, similar to \eqref{eq:gfe-unique} we find the identity
\[
\mathcal{A} g(t,w) = \e^{- \alpha\e^{\theta t} w} \e^{\int_0^t \kappa(\alpha \e^{\theta r}) \dd r }  \mathcal{L} f(y)|_{y= \e^w}, \qquad t\in [0,T_0], w\in \R, 
\]
where $g(t,w):=f(\e^w) \e^{- \alpha\e^{\theta t} w} \e^{\int_0^t \kappa(\alpha \e^{\theta r}) \dd r }$.  
Appealing to Proposition~\ref{Ch4:prop:mtoo} ends the proof. 
\end{proof}


\subsection{Convergence of OU type growth-fragmentations}\label{Ch4:sec:CV}
For every $n\in \bar{\N}:= \N \cup \{\infty\}$, let $\Xs_n$ be an OU type growth-fragmentation with characteristics $(\sigma_n, c_n, \nu_n, \dOU_n)$ starting from $\Xs_n(0)=(1,0,\ldots)$ and $\kappa_n$ be its cumulant. We establish the following convergence result. 

\begin{thm}\label{thm:CV}
 Suppose that   
 \begin{equation}\label{eq:CV-0}
 \nu_n\big( (0,0,\ldots) \big) = 0 ~\text{for all } n\in \bar{\N}, 
 \end{equation}
 that
  \begin{equation}\label{eq:CV-dOU}
  \lim_{n\to \infty} \dOU_n = \dOU_{\infty}, 
  \end{equation}
that
 \begin{equation}\label{eq:CV-c}
   \lim_{n\to \infty} (c_n+ \sigma_n^2/2) = c_{\infty}+ \sigma_{\infty}^2/2, 
 \end{equation}
and that there is the weak convergence of finite measures on $\Sd$
\begin{equation}\label{eq:CV-nu}
  \sigma_n^2 \delta_{\mathbf{1}} (\dd \sd) + (1-s_1)^2 \nu_n(\dd \sd) \underset{n\to \infty}{\Longrightarrow} \sigma_{\infty}^2 \delta_{\mathbf{1}} (\dd \sd) + (1-s_1)^2 \nu_{\infty}(\dd \sd). 
\end{equation}
Write $\bar{\dOU}:=\sup_{n\in \bar{\N}} \dOU_n<\infty$, then for every $T\geq 0$ and $q>2 (\e^{\bar\dOU T}\vee 1)$, there is the weak convergence
\begin{equation*}
  \Big(\Xsd_{n}(t), ~ t\in [0,T]\Big)  \underset{n\to \infty}{\Longrightarrow} \Big(\Xsd_{\infty}(t),~ t\in [0,T]\Big),
\end{equation*}
in the space $D([0,T], \lp[q])$ of c\`adl\`ag functions with values in $\lp[q]$ endowed with the Skorokhod topology. As a consequence, the weak convergence 
\begin{equation*}
\Xsd_{n}  \underset{n\to \infty}{\Longrightarrow} \Xsd_{\infty},
\end{equation*}
holds in the space $D(\R_+, \co)$ of c\`adl\`ag functions with values in $\co$ endowed with the Skorokhod topology. 
\end{thm}
This result generalizes Theorem~2 in \cite{Bertoin:CF}, which deals with the case $\dOU_n \equiv 0$ for every $n\in \bar{\N}$; the assumptions \eqref{eq:CV-c} and \eqref{eq:CV-nu} are inherited from there.\footnote{There is a typo in Theorem~2 in \cite{Bertoin:CF} for the condition \eqref{eq:CV-c}.}
The condition \eqref{eq:CV-0} is a minor technical assumption that makes our arguments less cumbersome. 

\begin{rem}
	The reason for which we consider the space $\lp[q]$ with $q>2 (\e^{\bar\dOU T}\vee 1)$ is as follows. 
	Recall that $[2,\infty)\subset \mathrm{dom}(\kappa_n)$ for all $n\in \bar{\N}$, then $\Xsd_n(t)\in \lp[2 (\e^{\bar\dOU T}\vee 1)]$ (by Theorem~\ref{Ch4:thm:mean}) for every $t\in [0,T]$. 
	We further need to enlarge the state space to $\lp[q]$ with $q>2 (\e^{\bar\dOU T}\vee 1)$, so as to ensure that $(\Xsd_n(t))_{n\in \N}$ is tight in $\lp[q]$, which does not necessarily hold with $q=2 (\e^{\bar\dOU T}\vee 1)$. See the proof of Lemma~\ref{lem:CV-finite} below for details. 
\end{rem}

Before tackling the proof of Theorem~\ref{thm:CV}, we point out several evidences that suggest its validity. First, \eqref{eq:CV-c} and \eqref{eq:CV-nu} yield the convergence of the cumulant  
\begin{equation}\label{eq:CV-kappa}
 \lim_{n\to \infty}   \kappa_n(p) = \kappa_{\infty}(p),\quad \text{for all } p>2.  
\end{equation}
However, this convergence does not necessarily hold for $p=2$. 
Second, we have the convergence of the selected fragments defined as in Lemma~\ref{Ch4:lem:selected}.
Indeed, one easily deduces from \eqref{eq:CV-c} and \eqref{eq:CV-nu} the convergence of the Laplace exponents \eqref{Ch4:eq:Phi-selected}:
\begin{equation}\label{eq:CV-Phi}
\lim_{n\to \infty}   \Phi_{n,*}(p) = \Phi_{\infty,*}(p),\quad \text{for all } p\geq 0.   
\end{equation}
Then the convergence of the selected fragments is a consequence of the following lemma. 

\begin{lem}\label{lem:CV-Z}
For every $n\in \bar{\N}$, let $Z_n$ be an OU type process with characteristics $(\Phi_{n,*}, \dOU_n)$ starting from $Z_n(0)=0$. Suppose that \eqref{eq:CV-dOU} and \eqref{eq:CV-Phi} hold. Then there exists a coupling of $(Z_n, n\in \bar{\N})$, such that for every $t\geq 0$
\[ \lim_{n\to \infty} \sup _{s\in [0,t]}|Z_n(s) - Z_{\infty}(s)| =0, ~\text{in probability}.\]
\begin{proof}
Recall from \eqref{Ch4:eq:solOU} that $Z_n$ is a stochastic integral: 
\[Z_n(t)= \int_0^t \e^{-\dOU_n (t-s)} \dd\xi_n(s), \qquad t\geq 0,\]
where $\xi_n$ is a L\'evy process with Laplace exponent $\Phi_{n,*}$. 
We first observe that there exists a coupling of L\'evy processes $(\xi_n)_{n\in\bar{\N}}$, such that for every $t\geq 0$
\[ \lim_{n\to \infty} \sup _{s\in [0,t]}|\xi_n(s) - \xi_{\infty}(s)| =0, \quad \text{ in probability};\]
see e.g. \cite[Theorem~15.14 and 15.17]{Kallenberg}. 
Therefore, an application of \cite[Theorem~5]{Jakubowski} leads to the claim, if  $(\xi_n)_{n\in \N}$ satisfy the so-called \emph{condition UT}. To check the \emph{condition UT}, we shall use \cite[Lemme~3.1]{JMP:cv}. Consider $\xi_n^1(t) := \xi_n(t) - \sum_{|\Delta \xi_n (s)| >1}\Delta \xi_n (s)$. Then $b_n^1:= \Exp{\xi_n^1(1)}$ is finite, and $M_n^1(t):= \xi_n^1(t) - b_n^1 t$ is a martingale. In other words, the canonical decomposition of the special semimartingale $\xi_n^1$ is given by  
\[\xi_n^1(t) =b_n^1 t   + M_n^1(t).  \]
The family of the variations of the processes $(b_n^1 t)_{t\geq 0}$ is clearly tight, then it follows from \cite[Lemme~3.1]{JMP:cv} that $(\xi_n)$ satisfy the \emph{condition UT}. 
\end{proof}
\end{lem}
The rest of this subsection is devoted to the proof of Theorem~\ref{thm:CV}. By Prokhorov's theorem (see e.g. \cite[Section~5]{Billingsley:Convergence}), we shall prove the weak convergence of finite dimensional distributions and the tightness. 
In the remaining of this subsection, we fix $\bar{\dOU}:=\sup_{n\in \bar{\N}} \dOU_n<\infty$, $T\geq 0$ and $q>2 (\e^{\bar\dOU T}\vee 1)$.  

\paragraph{Convergence of finite dimensional distributions}
The proof of the weak convergence of finite dimensional distributions proceeds as Lemma~7 in \cite{Bertoin:CF}, though we overcome non--trivial difficulties which require new estimates. 
Consider for every $n\in \bar{\N}$ and $\levOU \geq 0$ the truncated OU type growth-fragmentation $\cutOU{\Xs}_n$. Recall that $\cutOU{\Xs}_n$ corresponds to an OU type branching Markov process $\cutOU{\Zs}_n$ with characteristics $(\sigma_n,c_n, \cutOU{\mu}_n, \dOU_n)$, where $\cutOU{\mu}_n$ is the image of $\nu$ by the map $(s_1, s_2, \ldots)\mapsto \cutOU{(\log s_1, \log s_2, \ldots)}$ as in \eqref{Ch4:eq:cut}. 

\begin{lem}[{\cite[Lemma~6]{Bertoin:CF}}]\label{lem:CV-cutnu}
Suppose that \eqref{eq:CV-0} and \eqref{eq:CV-nu} hold. Then for every $\levOU\geq 0$, there is  the weak convergence of finite measures on $\Rd$
\begin{equation*}
  \sigma_n^2 \delta_{(0,-\infty,\ldots)} (\dd \rr) + (1-\e^{r_1})^2 \cutOU{\mu}_n(\dd \rr) \underset{n\to \infty}{\Longrightarrow} \sigma_{\infty}^2 \delta_{(0,-\infty,\ldots)} (\dd \rr) + (1-\e^{r_1})^2 \cutOU{\mu}_{\infty}(\dd \rr), 
\end{equation*}
and 
\[  \cutOU{\mu}_n(\cdot~|\Rdtwo)\underset{n\to \infty}{\Longrightarrow} \cutOU{\mu}_{\infty}(\cdot~|\Rdtwo).  \]
\end{lem}

 These relations lead to the following convergence.  
\begin{lem}\label{lem:CV-cut}
Suppose that \eqref{eq:CV-0}, \eqref{eq:CV-dOU}, \eqref{eq:CV-c} and \eqref{eq:CV-nu} hold. Then for every $\levOU \geq 0$, there exists a coupling of $(\cutOU{\Xsd}_n)_{n\in\bar{\N}}$, such that for every $t\geq 0$ and $p\geq 2$, 
\[ \lim_{n\to \infty} \|\cutOU{\Xsd}_n (t) -\cutOU{\Xsd}_{\infty}(t) \|_{\ell^p} = 0 \quad \text{in probability}. \]
\begin{proof}
	Recall that in the construction of $\cutOU{\Zs}_n$ by Definition~\ref{Ch4:dfn:BOUPF}, each particle $u\in \U$ is born at time $b_{n,u}\geq 0$ with initial position $a_{n,u}$, and then moves according to an OU type process $\cutOU{Z}_{n,u}$ with characteristics $(\cutOU{\psi}_n,\dOU_n)$, where $\cutOU{\psi}_n$ is given by \eqref{Ch4:eq:Psi}. After an exponential time $\cutOU{\lambda}_{n,u}$ with parameter $\cutOU{\nu}_n(\Sdtwo)$, it splits into at most $\lceil \e^{\levOU} \rceil$ particles (see \eqref{eq:el}) whose relative positions are $(\cutOU{\Delta a}_{n,ui}, i\in \N)$, distributed according to $\cutOU{\nu}_n(\cdot~|\Sdtwo)$. 
	 We shall prove that there exists a coupling of $(\cutOU{\Zs}_n)_{n\in\bar{\N}}$, such that
	 the following sequences indexed by $\U$
	\[
	\Big(
	\ind{\cutOU{b}_{n,u}\leq t<\cutOU{b}_{n,u}+\cutOU{\lambda}_{n,u}}\exp\big( \e^{-\dOU_n (t-\cutOU{b}_{n,u})} \cutOU{a}_{n,u} + \cutOU{Z}_{n,u}(t-\cutOU{b}_{n,u}) \big), \quad u\in \U
	\Big)
	\]
	 converges in probability as $n\to \infty$, for $\ell^p$-distance. 
	 Then the claim follows since the rearrangement of sequences in decreasing order decreases the $\ell^p$-distance.

For every $u\in \U$, we may assume by Lemma~\ref{lem:CV-cutnu} and Skorokhod representation theorem that the random variables $\cutOU{\lambda}_{n,u}$, $\cutOU{\Delta a}_{n,u}$ are coupled in such a way that  
\begin{equation}\label{eq:CV-lambda}
\lim_{n\to \infty} \cutOU{\lambda}_{n,u} = \cutOU{\lambda}_{\infty,u}, \quad \text{a.s.}
\end{equation}
and 
\[
\lim_{n\to \infty} \cutOU{\Delta a}_{n,ui} = \cutOU{\Delta a}_{\infty,ui}, \quad \text{for all } i\in \N,~\text{a.s.}
\]
We further deduce from \eqref{eq:CV-c} and Lemma~\ref{lem:CV-cutnu} that  
$\lim_{n\to \infty} \cutOU{\psi}_n(p) = \cutOU{\psi}_{\infty}(p)$ for every $p\geq 0$. 
Using Lemma~\ref{lem:CV-Z} leads to 
\[
\lim_{n\to \infty}  \cutOU{Z}_{n,u}(s) = \cutOU{Z}_{\infty,u}(s), \qquad \text{for all } s>0,~\text{a.s.}
\]
Therefore, for every $u\in \U$, we have
\[
\lim_{n\to \infty} \exp\big(-\dOU_n (t-\cutOU{b}_{n,u})\big) \cutOU{a}_{n,u} + \cutOU{Z}_{n,u}(t-\cutOU{b}_{n,u}) 
=\exp\big(-\dOU_{\infty} (t-\cutOU{b}_{\infty,u})\big) \cutOU{a}_{\infty,u} + \cutOU{Z}_{\infty,u}(t-\cutOU{b}_{\infty,u}), \quad~\text{a.s.}
\]
Denote the set of vertices alive at time $t\geq 0$ by $V_{n,t}\subset{\U}$. Observe that $V_{n,t}$ is almost surely a finite set; furthermore, it follows from \eqref{eq:CV-lambda} that $V_{n,t}$ coincides with $V_{\infty,t}$ with high probability. Summarizing, we have completed the proof. 
\end{proof}
\end{lem}

We also need the following estimation. 
\begin{lem}\label{lem:CV-err}
For every $t\geq 0$ and $p\geq 2 (\e^{\bar\dOU t}\vee 1)$, there is 
\[ 
\lim_{\levOU \to \infty} \sup_{n\in \bar{\N}} \Exp{ \| \Xsd_n(t) -\cutOU{\Xsd}_n(t) \|_{\ell^p}^p}  = 0
\]
\begin{proof}
We deduce from \eqref{Ch4:eq:Xl-X} and Theorem~\ref{Ch4:thm:mean} that 
\[
\Exp{ \| \Xsd_n(t) -\cutOU{\Xsd}_n(t) \|_{\ell^p}^p } 
\leq  K_n(p,t) - \cutOU{K}_n(p,t) =K_n(p,t) \bigg( 1-  \exp \Big(\int_0^t- (\kappa_n(p\e^{-\dOU_n r}) -\cutOU{\kappa}_n(p\e^{-\dOU_n r})) \dd r \Big) \bigg),
\]
where $K_n(p,t):= \exp\big(\int_0^t \kappa_n(p \e^{-\dOU_n r}) dr\big)$ and $\cutOU{K}_n(p,t):= \exp\big(\int_0^t \cutOU{\kappa}_n(p \e^{-\dOU_n r}) \dd r\big)$. 
Since for every $\sd =(s_1, s_2, \ldots) \in \Sd$, there is
\[
 \sum_{i=2}^{\infty} \ind{s_i \leq \e^{-\levOU}} s_i^{p\e^{-\dOU_n r}} \leq  \e^{-\levOU (p \e^{-\dOU_n r} -2)} \sum_{i=2}^{\infty} s_i^{2}\leq \e^{-\levOU (p \e^{-\dOU_n r} -2)}(1-s_1)^2,
\]
we have
\[
\kappa_n(p\e^{-\dOU_n r}) -\cutOU{\kappa}_n(q\e^{-\dOU_n r}) = \int_{\Sd} \sum_{i=2}^{\infty} \ind{s_i \leq \e^{-\levOU}} s_i^{p\e^{-\dOU_n r}} \nu_n(\dd \sd) \leq  \e^{-\levOU (p \e^{-\dOU_n r} -2)} \int_{\Sd} (1-s_1)^2 \nu_n(\dd \sd).
\]
It follows that 
\[
\Exp{ \| \Xsd_n(t) -\cutOU{\Xsd}_n(t) \|_{\ell^p}^p } 
\leq  K_n(p,t) \bigg( 1-  \exp \Big(- \int_{\Sd} (1-s_1)^2 \nu_n(\dd \sd)\int_0^t \e^{-\levOU (p\e^{-\dOU_n r} -2)} \dd r \Big) \bigg).
\]
As $p > 2 (\e^{\bar \dOU t} \vee 1)$, we have $\inf_{n\in \bar\N, r\in [0,t]} (p \e^{-\dOU_n r} -2) >0$. We also deduce from \eqref{eq:CV-nu} and \eqref{eq:CV-kappa} that 
\begin{equation}\label{eq:nu-bdd}
\sup_{n\in \bar\N} \int_{\Sd} (1-s_1)^2 \nu_n(\dd \sd)<\infty, 
\end{equation}
and that  
\begin{equation}\label{eq:kappa-bdd}
\sup_{n\in \bar\N}K_n(p,t)\leq \sup_{n\in \bar\N} \exp \bigg(\int_0^t \big|\kappa_n(p \e^{-\dOU_n s})\big| \dd s \bigg)  <\infty.
\end{equation}
Then the claim follows.  
\end{proof}  
\end{lem}

We are now ready the prove the weak convergence of finite-dimensional distributions. 
\begin{lem}\label{lem:CV-finite}
Suppose that \eqref{eq:CV-0}, \eqref{eq:CV-dOU}, \eqref{eq:CV-c} and \eqref{eq:CV-nu} hold, then Theorem~\ref{thm:CV} holds for finite-dimensional distributions in $\lp[q]$.
\end{lem}
	\begin{proof}
For simplicity, we shall only establish the convergence for one-dimensional; similar arguments hold for multi-dimensional case. 
               
		We first claim that for $q'\in \big(2 (\e^{\bar{\theta}T} \vee 1) ,q\big)$, the set 
		\[
		B_r:= \big\{ \xs\in \lp[q]:~ \|\xs\|_{\ell^{q'}}\leq r  \big\},
		\]
		is a compact subset in $\lp[q]$. Indeed, for any sequence in $B_r$, we may use the diagonal procedure to extract a subsequence that converges pointwisely, and the limit belongs to $B_r$ due to Fatou's lemma. Since $B_r$ is equisummable in $\ell^q$ (because $q'<q$), the convergence also holds for $\ell^q$-distance. 
Next, it follows from Theorem~\ref{Ch4:thm:mean} that 
		\[
                \Prob{\Xsd_n(t)\not\in B_r} \leq r^{-q'} \Exp{\|\Xsd_n(t) \|_{\ell^{q'}}^{q'}} 
                = r^{-q'} \exp\big(\int_0^t \kappa_n(q' \e^{-\dOU_n r}) \dd r\big), \qquad t\in [0,T].                
                \]
                We hence deduce from \eqref{eq:kappa-bdd} that the sequence $(\Xsd_n(t), n\in \bar\N)$ is tight in $\lp[q]$. 

So it remains to prove the uniqueness of the limit of a converging subsequence. Let $k\in\N$ and $F: \R^{k}_+ \to [0,1]$ be a continuous function. For every $\xs=(x_1, x_2, \ldots) \in \lp[q]$, write $F(\xs):=F(x_1, \ldots, x_k)$. Then $F$ is continuous on $\lp[q]$. We shall prove for every $t\in [0,T]$ that
		\[ 
		\lim_{n\to \infty} \Exp{F(\Xsd_n(t) )} = \Exp{F(\Xsd_{\infty}(t) )}.
		\]
		If this holds for every $k\in \N$ and such function $F$, then we deduce the uniqueness of the limit. 

For every $\levOU\geq 0$ there is
\begin{align*}
& 	\big|\Exp{F(\Xsd_n(t) )-F(\Xsd_{\infty}(t) )}\big|\\
&\leq \left|\Exp{F(\cutOU{\Xsd}_n(t) )-F(\cutOU{\Xsd}_{\infty}(t) )}\right|+ \left|\Exp{F(\Xsd_n(t) )-F(\cutOU{\Xsd}_{n}(t) )}\right| +\left|\Exp{F(\Xsd_{\infty}(t) )-F(\cutOU{\Xsd}_{\infty}(t) )}\right|.
\end{align*}
Let us estimate these three terms. 
Fix an arbitrarily small $\epsilon>0$. By the tightness of $(\Xsd_n(t), n\in \bar\N)$ we may choose $r>0$ large enough such that 
\begin{equation*}
\Prob{\Xsd_n(t)\not\in B_r}<\epsilon \qquad  \text{ for every } n\in \bar\N.
\end{equation*}
Note that if $\Xsd_n(t)\in B_r$, then $\cutOU{\Xsd}_n(t)\in B_r$ for every $\levOU\geq 0$. So we have 
\[
\Prob{\cutOU{\Xsd}_n(t)\not\in B_r} \leq \Prob{\Xsd_n(t)\not\in B_r}< \epsilon.
\]
As $F$ is uniformly continuous on the compact subset $B_r$ in $\lp[q]$, there exists $\eta>0$ such that 
\[
|F(\xs) - F(\xs')| <\epsilon, \qquad \text{for all } \xs,\xs' \in B_r \text{ with } \|\xs-\xs'\|_{\ell^q}<\eta.
\]
Using Lemma~\ref{lem:CV-err} and Markov inequality, we next choose $\levOU$ large enough such that 
\[
\sup_{n\in \bar{\N}} \Prob{ \| \Xsd_n(t) -\cutOU{\Xsd}_n(t) \|_{\ell^q}\geq \eta} \leq \epsilon.
\]
We hence deduce that
\begin{align*}
&\left|\Exp{F(\Xsd_n(t) )-F(\cutOU{\Xsd}_{n}(t) )}\right| \\
&\leq \Prob{\cutOU{\Xsd}_n(t)\not\in B_r} +\Prob{\Xsd_n(t)\not\in B_r}+\Prob{\| \Xsd_n(t) -\cutOU{\Xsd}_n(t) \|_{\ell^q}\geq \eta}+\epsilon <4 \epsilon, \quad \text{for all}~ n\in \bar{\N}.
\end{align*}
By Lemma~\ref{lem:CV-cut}, we may further choose $n$ large enough such that $\Prob{\| \cutOU{\Xsd}_n(t) -\cutOU{\Xsd}_{\infty}(t) \|_{\ell^q}\geq \eta}< \epsilon$.
Applying the same arguments to $\left|\Exp{F(\cutOU{\Xsd}_n(t) )-F(\cutOU{\Xsd}_{\infty}(t) )}\right|$ completes the proof. 
 	\end{proof}

\paragraph{Tightness}
We finally complete the proof of Theorem~\ref{thm:CV} by checking Aldous' tightness criterion (see e.g. Theorem~16.11 in \cite{Kallenberg}). 
\begin{lem}
Let $(h_n, n\in \N)$ be a sequence of constants with $h_n>0$ and $\lim_{n\to \infty} h_n = 0$, and $(\tau_n, n\in \N)$ be a sequence of $\Xsd_n$-stopping times with $\tau_n <T$ almost surely. Suppose that \eqref{eq:CV-0}, \eqref{eq:CV-dOU}, \eqref{eq:CV-c} and \eqref{eq:CV-nu} hold, then we have for every $q> 2 (\e^{\bar\dOU T} \vee 1) $ with $\bar{\dOU}:= \sup_{n\in\bar{\N}} \dOU_n$,
\[ \lim_{n\to \infty} \|\Xsd_n (\tau_n) -\Xsd_{n}(\tau_n + h_n) \|_{\ell^{q  }} = 0 \quad \text{in probability}. \]
\begin{proof}
Denote $\Xsd_n (\tau_n ) := (X_{n,1}(\tau_n), X_{n,2}(\tau_n),\ldots)$ and 
\[
\Xsd_n (\tau_n )^{\e^{-\dOU_n h_n}}:=(X_{n,1}(\tau_n)^{\e^{-\dOU_n h_n}}, X_{n,2}(\tau_n)^{\e^{-\dOU_n h_n}},\ldots).
 \]
An elementary inequality leads to
\begin{align}
&\Exp{\big\|\Xsd_n (\tau_n )- \Xsd_n (\tau_n +h_n )\big\|_{\ell^{q }}^q} \nonumber \\
&\leq
2^{q-1} \bigg( \Exp{  \big\|\Xsd_n (\tau_n ) -\Xsd_n (\tau_n )^{\e^{-\dOU_n h_n}} \big\|_{\ell^{q}}^q  +  \big\|\Xsd_n (\tau_n )^{\e^{-\dOU_n h_n}} - \Xsd_n (\tau_n + h_n )\big\|_{\ell^{q}}^q  }\bigg). \label{eq:tight} 
\end{align}
We shall evaluate the two expected values in \eqref{eq:tight} respectively. Let us start with the first one. 
Applying the mean value theorem to the function $ x \mapsto  X_{n,i} (\tau_n )^x$, we obtain that 
\[
  \big| X_{n,i} (\tau_n ) -X_{n,i} (\tau_n )^{\e^{-\dOU_n h_n}} \big|  
\leq   \ind{X_{n,i} (\tau_n ) > 0} \max\Big( X_{n,i} (\tau_n ),X_{n,i} (\tau_n )^{\e^{-\dOU_n h_n}}  \Big) \Big|\log \big( X_{n,i} (\tau_n )\big) \Big| \big|1 - \e^{-\dOU_n h_n}\big|. 
\]
Denote $c_I:= \inf_{n\in \N} \e^{-\dOU_n h_n}$ and $c_S:= \sup_{n\in \N} \e^{-\dOU_n h_n}$, then 
 \[
\max\Big( X_{n,i} (\tau_n ),X_{n,i} (\tau_n )^{\e^{-\dOU_n h_n}}\Big) \leq X_{n,i} (\tau_n )+X_{n,i} (\tau_n )^{c_I} + X_{n,i} (\tau_n )^{c_S}. 
\]
As $h_n \to 0$ and $q >2 (\e^{\bar\dOU T} \vee 1)$, without loss of generality, we may assume that $\sup_{n\in\N} |h_n|$ is small enough such that $q c_I  >2 (\e^{\bar\dOU T} \vee 1)$. 
Then fix $\delta>0$ such that $q(c_I\wedge 1) (1-\delta) > 2 (\e^{\bar\dOU T} \vee 1)$. 
It is elementary to see that there exists $c_{\delta}>0$ such that $|\log x| \leq c_{\delta}(x^{\delta} + x^{-\delta})$ for all $x>0$, then we have
\[
\big\|\Xsd_n (\tau_n ) -\Xsd_n (\tau_n )^{\e^{-\dOU_n h_n}} \big\|_{\ell^{q}}^q  \leq c_{\delta} \big|1 - \e^{-\dOU_n h_n}\big| \sum_{k=1}^{6} \Big(\sum_{i=1}^{\infty}   X_{n,i} (\tau_n )^{q_k}\Big),
\]
where $\{q_k\}$ are constants $\{ qc_I \pm q\delta, q\pm q\delta, q c_S \pm q\delta\}$. These constants are all greater than $2 (\e^{\bar \dOU T} \vee 1)$ thanks to the choice of $\delta$, so we have martingales by Proposition~\ref{Ch4:prop:mart}. As $\tau_n <T$ a.s., using the optional stopping theorem to these martingales yields
 \[
\Exp{\big\|\Xsd_n (\tau_n ) -\Xsd_n (\tau_n )^{\e^{-\dOU_n h_n}} \big\|_{\ell^{q}}^q } 
\leq c_{\delta} \big|1 - \e^{-\dOU_n h_n}\big| \sum_{k=1}^6 \exp \Big(\int_0^T \big|\kappa_n(q_k \e^{-\theta_n r})\big| \dd r\Big).  
\]
As $\theta_n h_n \to 0$ and \eqref{eq:kappa-bdd} holds, this leads  to 
 \[
\lim_{n\to \infty}
\Exp{\big\|\Xsd_n (\tau_n ) -\Xsd_n (\tau_n )^{\e^{-\dOU_n h_n}} \big\|_{\ell^{q}}^q } =0.
\]

We next proceed to the second term in \eqref{eq:tight}. From the strong Markov property (see Proposition~\ref{Ch4:prop:branching} and Remark~\ref{Ch4:rem:SMP}), we have that
\[ 
\Exp{ \|\Xsd_n (\tau_n )^{\e^{-\dOU_n h_n}}- \Xsd_n (\tau_n +h_n )\|_{\ell^{q }}^q } 
 = \Exp{ \|\Xsd_n (\tau_n )^{\e^{-\dOU_n h_n}}\|_{\ell^{q}}^q }  \Exp{ \|\Xsd_n (h_n )- (1, 0,\ldots)\|_{\ell^{q}}^q } .
 \]
 Again, as $\tau_n \leq T$ a.s., it follows from Proposition~\ref{Ch4:prop:mart} and the optional stopping theorem that 
\[ 
\Exp{ \|\Xsd_n (\tau_n )^{\e^{-\dOU_n h_n}}\|_{\ell^{q}}^q } \leq \exp \bigg(\int_0^T \big| \kappa_n(q\e^{-\dOU_n h_n} \e^{-\dOU_n s}) \big| \dd s \bigg)\leq \exp \bigg(\int_0^{T+h_n} \big| \kappa_n(q \e^{-\dOU_n s}) \big| \dd s \bigg).
\]
We hence deduce from \eqref{eq:kappa-bdd} that 
\begin{equation}\label{eq:tight-bdd}
 \sup_{n\in \N}\Exp{ \|\Xsd_n (\tau_n )^{\e^{-\dOU_n h_n}}\|_{\ell^{q}}^q } \leq \sup_{n\in \N}\exp \bigg(\int_0^{T+h_n} \big| \kappa_n(q \e^{-\dOU_n s}) \big| \dd s \bigg)<\infty.
\end{equation}
Write $\tilde{\Xsd}_n (h_n )$ for the sequence obtained from $\Xsd_n (h_n )$ by exchanging the selected fragment $X_{n,*} (h_n )$ (see Lemma~\ref{Ch4:lem:selected}) and the largest one. Rearranging sequences in decreasing order reduces the $\ell^q$-distance, so
 \[\Exp{ \|\Xsd_n (h_n )- (1, 0,\ldots)\|_{\ell^{q}}^q }  \leq \Exp{ \|\tilde{\Xsd}_n (h_n )- (1, 0,\ldots)\|_{\ell^{q}}^q }.\] 
 Furthermore, it follows from Lemma~\ref{Ch4:lem:selected} and Theorem~\ref{Ch4:thm:mean} that
\[
\Exp{ \|\tilde{\Xsd}_n (h_n )- (1, 0,\ldots)\|_{\ell^{q}}^q }  
= \Exp{ |X_{n,*} (h_n )- 1|^q} + \exp \bigg(\int_0^{h_n} \kappa_n(q \e^{-\dOU_n s}) \dd s \bigg)-\exp \bigg(\int_0^{h_n}\Phi_{n,*}(q \e^{-\dOU_n s})   \dd s \bigg).
\]
On the one hand, for an even integer $N>q$, by H\"older's inequality we have  
\begin{align*}
 \Exp{ |X_{n,*} (h_n )- 1|^q} \leq \Exp{ |X_{n,*} (h_n )- 1|^N}^{q/N} 
&= \Exp{ \sum_{k=0}^N \binom{N}{k} (-1)^{N-k}X_{n,*} (h_n )^k}^{q/N}\\
&= \bigg( \sum_{k=0}^N \binom{N}{k}(-1)^{N-k}\exp \Big( \int_0^{h_n}   \Phi_{n,*}(k\e^{-\dOU_n s}) \dd s \Big)\bigg)^{q/N}.
\end{align*}
Since $\lim_{n\to \infty} \Phi_{n,*}(p) =\Phi_{\infty,*}(p)$ for every $p\geq 0$, we deduce that 
\[\lim_{n\to \infty}\exp \Big( \int_0^{h_n}   \Phi_{n,*}(k\e^{-\dOU_n s}) \dd s \Big) = 1, \quad \text{for every}~k = 0,1,\ldots , N, \]
which leads to
 \[ \lim_{n\to \infty} \Exp{ |X_{n,*} (h_n )- 1|^q}  = 0. \]
On the other hand, for every $p\geq 2$, there is 
\[ \kappa_n(p) -\Phi_{n,*}(p) =\int_{\Sd} \sum_{i=2}^{\infty} s_i^p \dd \sd \leq\int_{\Sd} (1-s_1)^2 \nu_n( \dd \sd). \]
Then we have
\begin{align}
&\exp \bigg(\int_0^{h_n} \kappa_n(q \e^{-\dOU_n s}) \dd s \bigg)-\exp \bigg(\int_0^{h_n}\Phi_{n,*}(q \e^{-\dOU_n s})   \dd s \bigg)\nonumber \\
&\qquad \leq  \exp \bigg(\int_0^{h_n} \kappa_n(q \e^{-\dOU_n s}) \dd s \bigg) \bigg(1 - \exp\Big(- h_n \int_{\Sd}  (1-s_1)^2 \nu_n(\dd \sd) \Big)   \bigg). \label{eq:CV-tight-select}
\end{align} 
Since \eqref{eq:nu-bdd} and \eqref{eq:kappa-bdd} hold, then \eqref{eq:CV-tight-select} converges to $0$ as $n\to \infty$. We hence conclude that 
\[\lim_{n\to \infty}\Exp{ \|\Xsd_n (h_n )- (1, 0,\ldots)\|_{\ell^{q}}^q } =0.\]
This and \eqref{eq:tight-bdd} entail that
\[ 
\lim_{n\to \infty} \Exp{ \|\Xsd_n (\tau_n )^{\e^{-\dOU_n h_n}}- \Xsd_n (\tau_n +h_n )\|_{\ell^{q }}^q } = 0.
\]
We have completed the proof. 
\end{proof}
\end{lem}
%
%
%
%
 \subsection{A law of large numbers for the inward case}\label{Ch4:sec:LLN}
In this subsection we fix an OU type growth-fragmentation $\Xsd$ with characteristics $(\sigma, c, \nu, \dOU)$ and cumulant $\kappa$, and always suppose that $\Xsd$ is \emph{inward}, i.e. $\dOU>0$.  
We shall study the long-time asymptotic behavior of $\Xsd$.

Before stating our main results, Theorem~\ref{Ch4:thm:LLN}, let us introduce the required assumptions. 
We first suppose that the cumulant $\kappa$ satisfies
\begin{equation}\label{Ch4:eq:kappa0}
  \kappa(0)= \int_{\Sd}( \#\sd  -1 )\nu (\dd \sd) <\infty,
\end{equation}
where $\#\sd:= \sum_{i=1}^{\infty}\ind{s_i> 0}$. Denote 
\[\Sd_1:= \{\sd \in \Sd:~ s_1>0, s_2=s_3=\ldots =0 \},\]
then \eqref{Ch4:eq:kappa0} forces that $\nu(\Sdtwo)<\infty$. So the branching rate is finite and on average a finite number of child particles are generated in each splitting event. 
Denote the number of particles at time $t\geq 0$ by 
\begin{equation*}
  N(t):= \sum_{i=1}^{\infty} \ind{X_i(t)\neq 0}. 
\end{equation*}
Under condition \eqref{Ch4:eq:kappa0}, the process $(N(t),t\geq 0)$ is simply a \emph{branching process}; see e.g. \cite{AthreyaNey} for basic properties. In particular, it is finite at all time. 

We further suppose that 
 \begin{equation}\label{Ch4:eq:non-extinct}
\kappa(0) >0,
 \end{equation}
 which is known as the \emph{supercritical} condition for the branching process $N$. It is known (Theorem~III.4.1 in \cite{AthreyaNey}) that \eqref{Ch4:eq:non-extinct} is a sufficient and necessary condition such that the following \emph{non-extinction} event has strictly positive probability:
\begin{equation*}
  \big\{ N(t) >0 ~\text{for all}~ t\geq 0 \big\} .
\end{equation*}

 We next replace \eqref{Ch4:eq:kappa0} by a stronger condition
 \begin{equation}\label{Ch4:eq:gamma}
 \text{there exists } \gamma\in (1,2], \text{ such that }  \int_{\Sdtwo} (\#\sd)^{\gamma}  \nu (\dd \sd) <\infty.
 \end{equation}
The purpose of this assumption is to make use of the following well-known martingale convergence result.

\begin{lem} [{\cite[Theorem~5]{Biggins:1992}}]\label{Ch4:lem:Nt}
   Suppose that \eqref{Ch4:eq:non-extinct} and \eqref{Ch4:eq:gamma} hold. 
Then the martingale 
\[
M_t:= \e^{-\kappa(0)t} N(t) 
\]
 converges to a limit $M_{\infty}$ as $t\to \infty$,  almost surely and in $L^{\gamma}(\mathbb P)$. 
Furthermore, conditionally on non-extinction, the limit $M_{\infty}$ is strictly positive.
\end{lem}
In particular, Lemma~\ref{Ch4:lem:Nt} entails that $(M_t)_{t\geq 0}$ is bounded in $L^{\gamma}(\mathbb P)$, i.e. there exists $C_{\gamma}>0$ such that
\begin{equation}\label{Ch4:eq:bdd-gamma}
  \sup_{t\geq 0} \Exp{M_t^{\gamma}} <C_{\gamma}.
\end{equation} 
Note that \eqref{Ch4:eq:gamma} is also a necessary condition for $M_t$ to have finite $\gamma$-moment \cite[Corollary~III~6.1]{AthreyaNey}.

The last assumption is that 
  \begin{equation}\label{Ch4:eq:Pi0}
    \int_{\Sd} \sum_{i=1}^{\infty}\ind{0<s_i< \frac{1}{2}} \log (|\log s_i|) \nu(\dd \sd)<\infty.
  \end{equation}
To understand this condition, we recall from Lemma~\ref{Ch4:lem:Htransform} that, under condition $\kappa(0)<\infty$, 
\begin{equation}\label{Ch4:eq:Phi0}
  \Phi_{0}(q):= \kappa(q)-\kappa(0), \qquad q\geq 0
\end{equation}
is the Laplace exponent of some L\'evy process. Then we observe from Lemma~\ref{Ch4:lem:inv} that \eqref{Ch4:eq:Pi0} is the sufficient and necessary condition that an OU type process with characteristics $(\Phi_{0},\dOU)$ possesses a unique invariant probability distribution $\Pi_0$. Let $\tilde{\Pi}_0$ be the image of $\Pi_0$ by the map $y\mapsto \e^y$, so $\tilde{\Pi}_0$ is a probability measure on $\R_+$ with finite Mellin transform
	\begin{equation*}
	\int_{\R_+} x^q ~\tilde{\Pi}_0(\dd x) = \expp{\int_0^{\infty}\big(\kappa(\e^{-\dOU s} q) - \kappa(0) \big) \dd s}, \qquad q\geq 0.
	\end{equation*}

We now state the main result of this section. 
\begin{thm}\label{Ch4:thm:LLN}
  Suppose that \eqref{Ch4:eq:non-extinct}, \eqref{Ch4:eq:gamma} and \eqref{Ch4:eq:Pi0} hold.
Then for every bounded and continuous function $f$ on $\R_+$, 
  \begin{equation}\label{Ch4:eq:LLN}
    \lim_{t\to \infty} \e^{-\kappa(0) t} \sum_{i=1}^{\infty}\ind{X_i(t)>0} f( X_i(t)) =  \prm{\tilde{\Pi}_0}{f} M_{\infty} \qquad \text{in}~L^{\gamma}(\PP).
  \end{equation}
\end{thm}

\begin{rem}
It is known (Theorem~III.7.2 in \cite{AthreyaNey}) that the martingale $M_t$ converges to $M_{\infty}$ in $L^1(\PP)$ if and only if 
 \begin{equation*}
  \int_{\Sdtwo} \#\sd \ind{\#\sd>0}\log (\#\sd)  \nu (\dd \sd) <\infty.
 \end{equation*}
However, when \eqref{Ch4:eq:gamma} is replaced by this weaker condition, our proof of Theorem~\ref{Ch4:thm:LLN} cannot be extended to prove the convergence in $L^{1}(\PP)$. 
\end{rem}

As a consequence of Theorem~\ref{Ch4:thm:LLN}, we obtain a law of large numbers. 
\begin{cor}[Law of large numbers]\label{Ch4:cor:LLN}
  Suppose that \eqref{Ch4:eq:non-extinct}, \eqref{Ch4:eq:gamma} and \eqref{Ch4:eq:Pi0} hold. Then for every bounded and continuous function $f$ on $\R_+$, conditionally on non-extinction, there is 
  \begin{equation*}
    \lim_{t\to \infty} N(t)^{-1}\sum_{i=1}^{\infty}\ind{X_i(t)>0} f( X_i(t))=  \prm{\tilde{\Pi}_0}{f}\qquad \text {in probability}. 
  \end{equation*}
\end{cor}
\begin{proof}
Conditionally on non-extinction, $M_{\infty}$ is strictly positive. So it follows from Lemma~\ref{Ch4:lem:Nt} that 
\begin{equation*}
    \lim_{t\to \infty} \frac{\e^{\kappa(0)t}}{N(t)}=  M_{\infty}^{-1} \qquad \text{a.s.}
  \end{equation*}
Combining this and Theorem~\ref{Ch4:thm:LLN}, we deduce the claim.
\end{proof}
Theorem~\ref{Ch4:thm:LLN} and Corollary~\ref{Ch4:cor:LLN} should be compared with the law of large numbers for binary branching Gaussian OU process \cite{AdamczakMilos:CLT} and branching diffusions \cite{EHK:SLLN}, as well as convergence results for Crump-Mode-Jagers branching processes \cite{nerman1981CMJ, Jagers}. 

Another worthy-noting consequence of Theorem~\ref{Ch4:thm:LLN} is about the long-time asymptotic for the solutions to growth-fragmentation equations; see \cite{BertoinWatson2, MischlerScher} and references therein for related estimates. 
\begin{cor}\label{Ch4:cor:LLN-2}
  Suppose that \eqref{Ch4:eq:non-extinct}, \eqref{Ch4:eq:gamma} and \eqref{Ch4:eq:Pi0} hold. Let $(\rho_{\Xsd}(t),t\geq 0)$ be the solution to the growth-fragmentation equation \eqref{Ch4:eq:gf}. Then the probability measure $\e^{-\kappa(0)t}\rho_{\Xsd}(t)$ converges weakly to $\tilde{\Pi}_0$. Furthermore, $\tilde{\Pi}_0$ is a solution to the stationary equation: for every $f\in \Cinfty$, 
  \begin{equation}\label{Ch4:eq:gf-stationary}
  \prm{\tilde{\Pi}_0}{\mathcal L f} =  \kappa(0) \prm{\tilde{\Pi}_0}{ f} , 
  \end{equation}
where $\mathcal L$ is as in \eqref{Ch4:eq:Lc}. 
\end{cor}
\begin{proof}
Taking expectation to \eqref{Ch4:eq:LLN}, we deduce that $\e^{-\kappa(0)t}\rho_{\Xsd}(t)$ converges vaguely to $\tilde{\Pi}_0$. We also know that $\rho_{\Xsd}(t) \big(\R_+\big) = \Exp{N(t)} = \e^{\kappa(0)t}$, so $\e^{-\kappa(0)t}\rho_{\Xsd}(t)$ is indeed a probability measure and thus the convergence also holds weakly. 

It remains to prove that $\tilde{\Pi}_0$ is a solution to \eqref{Ch4:eq:gf-stationary}. Since $(\rho_{\Xsd}(t),t\geq 0)$ is a solution to \eqref{Ch4:eq:gf}, we easily check that 
\begin{equation*}
  \partial_t \prm{\e^{-\kappa(0)t} \rho_{\Xsd}(t)}{f} = -\kappa(0)\prm{\e^{-\kappa(0)t} \rho_{\Xsd}(t)}{f} +  \prm{\e^{-\kappa(0)t} \rho_{\Xsd}(t)}{\Lc f}. 
\end{equation*}
Letting $t\to \infty$, we conclude the claim. 
\end{proof}

\begin{qu}
A natural question is whether the convergence in Theorem~\ref{Ch4:thm:LLN} also holds \emph{almost surely}. 
The methods developed in \cite{EHK:SLLN} might be of use: that is, by first proving along lattice times, and then extending to continuous time.  
\end{qu}
\begin{qu}
It would be interesting to extend Theorem~\ref{Ch4:thm:LLN} to the case when \eqref{Ch4:eq:kappa0} does not hold. Taking into account of the many-to-one formula, this would require an in-depth examination of the affine Markov process given in Lemma~\ref{lem:affine}.  
\end{qu}


We now prove Theorem~\ref{Ch4:thm:LLN}. 
\begin{proof}[Proof of Theorem~\ref{Ch4:thm:LLN}]
Equivalently, we shall prove that for every bounded and continuous function $g$ on $\R$, we have the convergence
  \begin{equation*}
    \lim_{t\to \infty} \e^{-\kappa(0) t} \sum_{i=1}^{\infty}\ind{X_i(t)>0} g( \log X_i(t)) =  \prm{\Pi_0}{g} M_{\infty} \qquad \text{in}~L^{\gamma}(\PP). 
  \end{equation*}
For simplicity, denote
\[U_{t}:= \e^{-\kappa(0)t}\sum_{i=1}^{\infty}\ind{X_i(t)>0} g( \log X_i(t)), \qquad t\geq 0.\]
Let $(\F_t)_{t\geq 0}$ be the natural filtration of $\Xsd$, then it suffices to prove that
\begin{equation}\label{Ch4:eq:claim0}
  \lim_{t\to \infty} \sup_{s\geq 0} |U_{t+s}- \Expcond{U_{t+s}}{\F_t} |= 0\qquad  \text{in } L^{\gamma}(\PP),
\end{equation}
and that there exists a function $t\mapsto S(t)>0$ such that
\begin{equation}\label{Ch4:eq:claim}
  \lim_{t\to \infty}\Expcond{U_{t+S(t)}}{\F_t} = \prm{\Pi_0}{g}M_{\infty}\qquad  \text{in } L^{\gamma}(\PP).
\end{equation}

We start with \eqref{Ch4:eq:claim0}. Let $(\Xs^{(i)} :=(X_1^{(i)}(t),X_2^{(i)}(t), \ldots)_{t\geq 0},  i\geq 1)$ be i.i.d. copies of $\Xsd$, then using the Markov property (Proposition~\ref{Ch4:prop:branching}), we have for every $s\geq 0$ the identity in law:
\begin{equation}\label{Ch4:eq:Yi}
U_{t+s} - \Expcond{U_{t+s}}{\F_t} \eqdis \e^{-\kappa(0) (t+s)}\sum_{i=1}^{\infty}  \ind{X_i(t)>0}(Y_i(t,s)- \Expcond{Y_i(t,s)}{\F_t}),
\end{equation}
where
\begin{equation*}
  Y_i(t,s):= \sum_{j=1}^{\infty}\ind{X_j^{(i)}(s)> 0} g(\e^{-\dOU s} \log X_i(t) + \log X_j^{(i)}(s)).
\end{equation*}
Let us now recall a useful inequality \cite[Lemma~1]{Biggins:1992}: let $\gamma \in [1,2]$ and $(Z_i)_{i\in\N}$ be independent (but not necessarily identical) random variables with each $\Exp{Z_i}=0$, then for every $n\in \N \cup \left\{ \infty \right\}$ there is 
\begin{equation}\label{Ch4:eq:inqgamma}
  \Exp{\Big| \sum_{i=1}^{n} Z_i \Big|^{\gamma}} \leq 2^{\gamma} \sum_{i=1}^{n} \Exp{\Big|Z_i\Big|^{\gamma}}.
\end{equation}
Since $Z_i:= Y_i(t,s)- \Expcond{Y_i(t,s)}{\F_t}$ are independent conditionally on $\F_t$, applying \eqref{Ch4:eq:inqgamma} to \eqref{Ch4:eq:Yi}, we have
\begin{equation*}
   \Exp{ \Big|U_{t+s}- \Expcond{U_{t+s}}{\F_t} \Big|^{\gamma}} \leq  2^{\gamma}\e^{-\gamma \kappa(0) (t+s)} \sum_{i=1}^{\infty}  \Exp{ \ind{X_i(t)>0} \Big|Y_i(t,s) - \Expcond{Y_i(t,s)}{\F_t}\Big|^{\gamma}}.
\end{equation*}
For every $i\in \N$, using Jensen's inequality (the finite form) and then conditional Jensen's inequality, we find that 
\begin{align*}
&  \Exp{ \ind{X_i(t)>0} \Big|Y_i(t,s) - \Expcond{Y_i(t,s)}{\F_t}\Big|^{\gamma}} \\
&\qquad \leq 2^{\gamma -1} \Exp{ \ind{X_i(t)>0} (|Y_i(t,s)|^{\gamma}+ |\Expcond{Y_i(t,s)}{\F_t}|^{\gamma})}  \leq  2^{\gamma } \Exp{ \ind{X_i(t)>0} |Y_i(t,s)|^{\gamma}}.
\end{align*} 
By conditioning on $\F_t$ and using \eqref{Ch4:eq:bdd-gamma}, we deduce that 
\begin{equation*}
  \Exp{ \ind{X_i(t)>0} |Y_i(t,s)|^{\gamma}} \leq  \| g\|^{\gamma}_{\infty} \Exp{ \ind{X_i(t)>0} \Big( \sum_{j=1}^{\infty}\ind{X_j^{(i)}(s)> 0}\Big)^{\gamma} }\leq \| g\|^{\gamma}_{\infty}C_{\gamma}\e^{\gamma \kappa(0)s}   \Exp{ \ind{X_i(t)>0}}.
\end{equation*} 
Summarizing, for every $s,t>0$ we have 
\begin{equation*}
   \Exp{ |U_{t+s}- \Expcond{U_{t+s}}{\F_t} |^{\gamma}}\leq  2^{2\gamma} \| g\|^{\gamma}_{\infty}C_{\gamma}\e^{-(\gamma-1) \kappa(0) t}  , 
\end{equation*}
which converges to $0$ as $t\to \infty$, since $\gamma >1$. So we have justified \eqref{Ch4:eq:claim0}.

It remains to prove \eqref{Ch4:eq:claim}. Recall that
 $$\Expcond{U_{t+s}}{\F_{t}} = \e^{-\kappa(0) (t+s)}\sum_{i=1}^{\infty}  \ind{X_i(t)>0} \Expcond{Y_i(t,s)}{\F_t}.$$
 Let $\chi$ be an OU type process with characteristics $(\Phi_{0},\dOU)$, where $\Phi_0(\cdot):= \kappa(\cdot) - \kappa(0)$. Then applying the many-to-one formula (Proposition~\ref{Ch4:prop:mtoi} and Remark~\ref{rem:mto0}) to $\Expcond{Y_i(t,s)}{\F_t}$ yields
\begin{equation}\label{Ch4:eq:Yi-cond}
  \e^{-\kappa(0) s }\Expcond{Y_i(t,s)}{\F_{t}} =  \Exp{g(\e^{-\dOU s} \log x_i + \chi(s)) }\Big|_{x_i = X_i(t)}.
\end{equation}
Consider a family of increasing compact sets $K_t\uparrow (0,\infty)$, say $K_t: = [t^{-1}, t]$. 
 On the one hand, if we only consider those $X_i(t) \not\in K_t$, then it follows from \eqref{Ch4:eq:Yi-cond} that
 \[ \sup_{s\geq 0}\bigg|\e^{-\kappa(0)( t+s) }\sum_{i=1}^{\infty}  \ind{X_i(t)\not\in K_t}\Expcond{Y_i(t,s)}{\F_{t}} \bigg|
\leq \|g\|_{\infty} \e^{-\kappa(0) t }\sum_{i=1}^{\infty}  \ind{X_i(t)\not\in K_t} .\]
By the many-to-one formula, the right-hand-side has mean value $\|g\|_{\infty}\Prob{ \exp(\chi(t))\not\in K_t}$,
which converges to zero as $t\to \infty$ by Lemma~\ref{Ch4:lem:inv} and the Portmanteau theorem. As \eqref{Ch4:eq:bdd-gamma} holds, we have by the dominated convergence that
\begin{equation}\label{Ch4:eq:notKn}
 \lim_{t\to \infty}  \sup_{s\geq 0}\Big|\e^{-\kappa(0)(t+s) }\sum_{i=1}^{\infty}  \ind{X_i(t)\not\in K_t}\Expcond{Y_i(t,s)}{\F_{t}}\Big| = 0 \quad  \text{in } L^{\gamma}(\PP).
\end{equation}
On the other hand, since $g$ is uniformly continuous on any compact set $K$ on $\R$, we deduce by Lemma~\ref{Ch4:lem:inv} that the following convergence holds uniformly for all $z\in K$:
\[\lim_{s\to \infty}\Exp{g( \e^{-\dOU s} z+ \chi(s)) } =\prm{\Pi_0}{g}. 
\]
Then using \eqref{Ch4:eq:Yi-cond} and Lemma~\ref{Ch4:lem:Nt}, we can choose $S(t)>0$, depending on $K_t$, such that
\begin{equation}\label{Ch4:eq:inKn}
\lim_{t\to \infty}\e^{-\kappa(0) (t +S(t))}\sum_{i=1}^{\infty}  \ind{X_i(t)\in K_{t}} \Expcond{Y_i(t,S(t))}{\F_{t}} 
= \prm{\Pi_0}{g}M_{\infty} \quad \text{in } L^{\gamma}(\PP).
\end{equation}
Combining \eqref{Ch4:eq:notKn} and \eqref{Ch4:eq:inKn}, we then deduce \eqref{Ch4:eq:claim}, which completes the proof. 
\end{proof}

\section{Relations to Markovian growth-fragmentation processes}\label{Ch4:sec:Markovian}
In this section, we study \emph{Markovian growth-fragmentation processes} \cite{Bertoin:GF-Markovian} associated with exponential OU type processes. The main result (Proposition~\ref{Ch4:prop:MOU}) shows that such processes form a sub-family of our OU type growth-fragmentations.

\subsection{Markovian growth-fragmentations associated with exponential OU type processes}
Throughout this section, let $\xi$ be a spectrally negative L\'evy process with characteristics $(\sigma, c , \Lambda, k)$, $Z$ be an \emph{inward} OU type process with index $\theta>0$ driven by $\xi$ as in \eqref{Ch4:eq:solOU}, and
\[X(t):=  \exp(Z(t)), \qquad t\geq 0.\]
The assumption $\theta>0$ is made only for technical reasons; see Remark~\ref{rem:Mgf-o}.  
For every $x\geq 0$, write $\Pm_x$ for the law of $X$ starting from $X(0)=x$, with convention that $\Pm_0$ denotes the law of the process $X(t)\equiv 0$.
Let $\zeta := \inf\{t\ge 0\colon X(t) =0\}$ be the \emph{lifetime} of $X$, which can be infinite.

Recall that the Laplace exponent $\Phi$ of $\xi$ is given by \eqref{Ch4:eq:Phi}. 
We introduce $\kappa\colon [0,\infty) \to (-\infty, \infty]$ by
\begin{equation}\label{Ch4:eq:kappa}
  \kappa (q) := \Phi(q) + \int_{(-\infty,0)} (1-\e^{y})^q \Lambda  (\dd y), \qquad q\geq 0.
\end{equation}
Then $\kappa\geq \Phi$, $\kappa$ is convex and $\kappa(q)<\infty$ for all $q\geq 2$ because of \eqref{Ch4:eq:mLevy}.  The function $\kappa$ shall be referred to as the \emph{cumulant} of $\xi$ or $Z$ or $X$; we shall later (in Proposition~\ref{Ch4:prop:MOU}) see that $\kappa$ indeed plays a similar role as the cumulant of an OU type growth-fragmentation defined as in \eqref{Ch4:eq:kappaOU}. 
We emphasize that $\kappa$ does not characterize the law of $\xi$; see \cite[Lemma~2.1]{S:2}. The cumulant $\kappa$ also plays a crucial role in the study of self-similar growth-fragmentations; see \cite{Bertoin:GF-Markovian, S:2}.  

For future use, we statement the following property of $X$. Let $\eta \in (0,1)$ and $F: [0,\infty) \times [0,\infty) \to [ 0,\infty)$ be a function defined by 
\begin{equation}\label{eq:F}
 F(t,x):= x^{2\e^{\dOU t}}G_1(t)G_2(t), \qquad t\geq 0, x \ge 0,  
\end{equation} 
where  $G_1(t):= \exp \Big(- \int_{0}^{t} \Phi(2\e^{\dOU r})\dd r\Big)$  and 
$G_2(t):= \exp\Big(- \int_{0}^{t} \eta^{-1} \left(\kappa(2\e^{\dOU r}) - \Phi(2\e^{\dOU r}) \right)\dd r\Big)$. Note that $G_2$ is non-increasing. 

\begin{lem}\label{Ch4:lem:OU}
For every $r\in (0,\zeta)$, let $\Delta X(r):= X(r) - X(r-)$. 
Then for every $x>0$, $s, t \ge 0$, we have
\begin{equation}\label{Ch4:eq:H1OU}
  \Em_x\bigg[F(s+t,X(t))+ \sum_{0\leq r \leq t\wedge (\zeta-s)  } F(s+r,-\Delta X(r)) \bigg] \leq F(s,x).
\end{equation}
and 
\begin{equation}\label{Ch4:eq:HetaOU} 
   \Em_x\bigg[\sum_{0 \leq r < (\zeta-s)}F(s+r,-\Delta X(r))\bigg] \leq \eta F(s,x).
\end{equation}
\end{lem}

\begin{proof}
Applying \eqref{Ch4:eq:Lap} with $q=2 \e^{\dOU (t+s)}$, we have for every $s\geq 0$ that
  \begin{equation}\label{Ch4:eq:lem:OU}
    \Em_x\bigg[F(s+t,X(t)) \bigg] = x^{2 \e^{\dOU s}} \exp \left(\int_0^t \Phi(2 \e^{\dOU (t+s-r)} ) \dd r \right)G_1(s+t)G_2(s+t) =  x^{2 \e^{\dOU s}}G_1(s) G_2(s+t).
  \end{equation}
As \eqref{Ch4:eq:solOU} entails that $-\Delta X(r) = X(r-) ( 1- \e^{\Delta \xi(r)})$, applying the compensation formula (see e.g. \cite[Sec. O.5]{Bertoin:Levy}) to the Poisson point process $\Delta \xi$, we have that
\begin{align} 
    \Em_x\bigg[\sum_{0 \leq r\leq t}F(s+r,-\Delta X(r)) \bigg] 
&= \int_0^t  \Em_x\big[F(s+r,X(r)) \big] \dd r \int_{(-\infty,0)}  (1-\e^z)^{2 \e^{\dOU (s+r)} }\Lambda(\dd z)
  \nonumber \\
&= \int_0^t x^{2 \e^{\dOU s} }G_1(s) G_2(s+r) \left(\kappa(2\e^{\dOU (s+r)}) - \Phi(2\e^{\dOU (s+r)}) \right)\dd r \nonumber \\ 
&=  \eta x^{2 \e{\dOU s}}G_1(s) \left(G_2(s) - G_2(s+t)\right) \label{Ch4:eq:lem:OU1}
  \end{align}
where we have used \eqref{Ch4:eq:lem:OU} in the second equality. 
Adding \eqref{Ch4:eq:lem:OU} to \eqref{Ch4:eq:lem:OU1} and using the fact that $G_2$ is non-increasing, we obtain \eqref{Ch4:eq:H1OU}. Letting $t\to \infty$ in \eqref{Ch4:eq:lem:OU1}, we also have \eqref{Ch4:eq:HetaOU}.  
\end{proof}

Lemma~\ref{Ch4:lem:OU} enables us to list the jump times of $X$, excluding $\zeta$, as a sequence $(t_i, i\in \N)$ such that $(F(|\Delta X(t_i)|,t_i), i\in \N))$ is decreasing. In the sequel, \emph{the $i$-th jump time of $X$} shall always refer to the $i$-th element $t_i$ in this sequence.

A Markovian growth-fragmentation process associated with $X$ can be constructed by using the approach in \cite{Bertoin:GF-Markovian, S:2}.
We first construct a \emph{cell system driven by $X$}, which is a family of processes indexed by the Ulam-Harris tree $\U := \bigcup_{i=0}^{\infty}\N^i$,
$$\Xc:= (\Xc_u, ~ u\in \U),$$ 
where each $\Xc_u$ depicts the evolution of the size of the cell indexed by $u$ as time passes. Specifically, the ancestor cell $\emptyset$ is born at $b_{\emptyset}:= 0$ with initial size $1$, and the life career $\Xc_{\emptyset} = (\Xc_{\emptyset}(t), t\geq 0)$ is an OU type process of law $P_{1}$. 
The laws of the first generation $\N\subset \U$ are determined by the trajectory of $\Xc_{\emptyset}$: for $i\in \N$, say the $i$-th jump time of $\Xc_{\emptyset}$ occurs at time $t_i$ and has size $x_i:= -\Delta \Xc_{\emptyset}(t_i)$, we then set $b_i =t_i$ and build a sequence of conditional independent processes $(\Xc_i)_{i\in \N}$ with respective conditional distribution $P_{x_i}$. 
We stress that the lifetime $\zeta_{\emptyset}$ of $\Xc_{\emptyset}$ is excluded from the jump sequence $(t_i)$, and hence at time $\zeta_{\emptyset}$ no child is born. 
We continue in this way to construct higher generations recursively: For every individual $u\in \U$, the laws of her daughters are determined by the trajectory of $\Xc_u$: given $\Xc_u$ with lifetime $\zeta_u$, say the $i$-th jump of $\Xc_u$ is at time $t<\zeta_u$ with $y:= -\Delta \Xc_{u} (t)$, then its $i$-th daughter $ui$ is born at time $b_{ui} := b_u+ t$ and $ui$'s size process $\Xc_{ui}= (\Xc_{ui}(r), r\geq 0)$ has conditional distribution $P_{y}$, independent of the size processes of the other individuals in the same generation. By convention, if $t= \infty$ (which means that $\Xc_u$ has less than $i$ jumps), then we set the cell $ui$ as well as all its progeny to have degenerate life careers, i.e. for every $v\in \U$ we set $\Xc_{uiv} \equiv 0$ and $b_{uiv}= \infty$. 
The above description uniquely determine the law of the cell system $\Xc$, denoted by $\mathcal{P}$.

\begin{lem}
	Let $F$ be a function as in \eqref{eq:F}. For every $t\geq 0$,
	\[
 \mathcal{P}\bigg[ \sum_{u\in \U, b_u\leq t} F(t, \Xc_u(t-b_u)) \bigg] \leq 1.
\]
\end{lem}
\begin{proof}
	As \eqref{Ch4:eq:H1OU} holds, the claim follows from \cite[Lemma~3.2]{S:2}.
\end{proof}
In particular, this lemma implies that at every time $t\geq 0$, we can rank the sizes of the cells alive at $t$, i.e. $$ \multiset{\Xc_u(t-b_u)~:~ u\in \U, b_u\leq t < b_u +\zeta_u},$$
 in decreasing order and obtain a sequence in $\lp[2 \e^{\dOU t}]$ denoted by $\Xsd(t)$. 
We refer to $\Xsd = (\Xsd(t), t\geq 0)$ as a \emph{(Markovian) growth-fragmentation process driven by $X$}. Write $\Ps$ for the law of $\Xsd$ under $ \mathcal{P}$. 

By construction, the law of $\Xsd$ is determined by the law of $X$. However, in the following statement we find a family of OU type processes which give rise to the same (in finite-dimensional distributions) growth-fragmentation.

\begin{lem}\label{Ch4:lem:OU2}
Let $\tilde{Z}$ be an OU type process with characteristics $(\tilde{\Phi}, \dOU)$ and $\tilde{X}:= \exp(\tilde{Z})$. Suppose that $X$ and $\tilde{X}$ have the same cumulant $\kappa$, then the growth-fragmentations $\Xsd$ and $\tilde{\Xs} $, driven respectively by $X$ and $\tilde{X}$, have the same finite-dimensional distributions.   
\end{lem}

\begin{proof}
In order to apply Theorem~3.7 in \cite{S:2}, we introduce the following manipulation. Since $X$ and $\tilde{X}$ have the same cumulant $\kappa$, by Proposition~2.5 in \cite{S:2} we can build a pair of spectrally negative L\'evy processes $\xi$ and $\tilde{\xi}$ with respective Laplace exponents $\Phi$ and $\tilde{\Phi}$, such that $\xi$ is a \emph{switching transform} of $\tilde{\xi}$, see Lemma~2.2 in \cite{S:2} for the precise meaning. In particular, we have that the switching time $\tb:=\inf \left\{ t\geq 0:~ \xi(t)\neq \tilde{\xi}(t) \right\}$ is almost surely strictly positive and 
$ \e^{\xi(\tau)}+ \e^{\tilde{\xi}(\tau)} =  \e^{\xi(\tau-)}.$
We may assume $\log X$ and $\log \tilde{X}$ (both starting from $0$) are OU type processes associated respectively with $\xi$ and $\tilde{\xi}$ by \eqref{Ch4:eq:solOU}, then $\inf \left\{ t\geq 0:~ X(t)\neq \tilde{X}(t) \right\}$ is equal to $\tau$ and $X(\tau) +\tilde{X}(\tau) = X(\tau-)=\tilde{X}(\tau-)$.
Let $\tilde{X}'$ be an independent copy of $\tilde{X}$ and set
\[  \tilde{X}'' (t):=  \tilde{X} (t)\ind{t< \tb} + \tilde{X}(\tau)^{\exp(-\dOU (t-\tb))} \tilde{X}' (t-\tb)\ind{t\geq \tb} ,\qquad t\geq 0.\]
Using \eqref{Ch4:eq:solOU} and the strong Markov property of an OU type process, one easily checks that $\tilde{X}'' \eqdis \tilde{X} $ and further the couple $(X , \tilde{X}'')$ satisfies the following properties: 
 	\begin{enumerate}[label=$\mathbf{(B\arabic{*})}$, ref=$\mathbf{(B\arabic{*})}$,leftmargin=3.0em]
 		\item\label{B1} Let $\tb:= \inf \{t\geq 0: X (t) \neq  \tilde{X} (t) \}$. There is almost surely either $\tb=\infty$ or the identity
 		\begin{equation*}
 		X (\tb)+ \tilde{X}'' (\tb) = X (\tb-)= \tilde{X}'' (\tb-).
 		\end{equation*}
 		\item\label{B2}  (Asymmetric Markov branching property)
 		Conditionally on $\tb>t$, the process 
 		\[(X (r+t) X (t)^{-\exp(-\dOU t)}, \tilde{X}'' (r+t) \tilde{X}''(t)^{-\exp(-\dOU t)} )_{r\geq 0} \]
 		is a copy of $(X ,\tilde{X})$; conditionally given $\tb\leq t$, the two processes $(X (r+t)X (t)^{-\exp(-\dOU t)})_{r\geq 0}$ and $(\tilde{X}'' (r+t)\tilde{X}''(t)^{-\exp(-\dOU t)})_{r\geq 0} $ are independent, and have the laws of $X$ and $\tilde{X}''$ respectively. 
 	\end{enumerate}
Therefore, we find that $(X , \tilde{X}'' )$ is a  \emph{bifurcator} in the sense of Definition~3.7 in \cite{S:2}. 
Combining this and Lemma~\ref{Ch4:lem:OU}, we check that the conditions of Theorem~3.7 in \cite{S:2} are fulfilled, then it follows that $\Xsd$ and $\tilde{\Xs}$ have the same finite-dimensional distributions. 
\end{proof}

\subsection{Binary OU type growth-fragmentations and Markovian growth-fragmentations}
\begin{dfn}\label{Ch4:dfn:binary}
  A \emph{binary dislocation measure} $\nu$ is a sigma-finite measure on $\Sd$ that satisfies \eqref{Ch4:eq:nu} and has support on
\begin{equation}\label{Ch4:eq:b}
   \left\{ \sd\in \Sd~:~ s_1 +s_2=1  , s_3 = s_4 =\ldots = 0 \right\}\bigcup \{\partialX \}. 
\end{equation}
An OU type growth-fragmentation process is \emph{binary}, if its dislocation measure is binary. 
\end{dfn} 
In this subsection we study the relation between Markovian growth-fragmentations and OU type growth-fragmentation processes. 
We first observe that each binary OU type growth-fragmentation can be viewed as a Markovian growth-fragmentation in the following sense.  
\begin{prop}\label{Ch4:prop:OU1}
Let $\Xsd$ be a binary OU type growth-fragmentation with cumulant $\kappa$ defined in \eqref{Ch4:eq:kappaOU}, and $X_*$ be the size of the selected fragment of $\Xsd$. 
Then the cumulant of $X_*$ defined by \eqref{Ch4:eq:kappa} is also $\kappa$. Furthermore, $\Xsd$ is a Markovian growth-fragmentation associated with $X_*$.
\end{prop}
\begin{proof}
Recall from Lemma~\ref{Ch4:lem:selected} that $X_*$ evolves as the exponential of an OU type process with characteristics $(\Phi_*, \theta)$. Write $\Lambda_*$ for the L\'evy measure of $X_*$ and $\nu$ for the dislocation measure of $\Xsd$, then we have the identity
 \[
 \int_{(-\infty,0)} (1-\e^{y})^q \Lambda_*  (\dd y) =  \int_{\Sd\setminus \{(0,0,\ldots)\} } (1-s_1)^q \nu (\dd \sd) = \int_{\Sd} \bigg(\sum_{i=2}^{\infty} s_i^q \bigg)\nu (\dd \sd) ,  
 \]
 where the second equality follows from the fact that $\nu$ is binary.  
This leads to the first statement of the proposition.  

The proof of the second statement is an adaptation of arguments in \cite[Proof of Proposition~3]{Bertoin:GF-Markovian}.
For any $\levOU>0$, consider the truncated system $\cutOUd{\Xs}$ (see Lemma~\ref{Ch4:lem:cut} and Definition~\ref{Ch4:dfn:BOUP}). Note that the select fragment of $\cutOUd{\Xs}$ has the same size evolution $X_*$ as $\Xs$.  
Moreover, the system $\cutOUd{\Xs}$ has a discrete genealogical structure (corresponding to Definition~\ref{Ch4:dfn:BOUPF}).  
 Since the select fragment of $\cutOUd{\Xs}$ is obtained by keeping the larger child and discarding the smaller one at each dislocation, the dynamics of $\cutOU{\Xs}$ can be described in the following way. 
Let $\Pm_x$ be the law of the process $(x^{\exp(-\dOU t)} X_*(t))_{t\geq 0}$.
Initially, there is one fragment whose size evolves according to $X_*$ of law $\Pm_1$. 
By \eqref{Ch4:eq:cut}, the first dislocation of the system $\cutOUd{\Xs}$ happens at the first time $t\geq 0$ when $\frac{|\Delta X_*(t)|}{X_*(t-)} >\e^{-\levOU}$, and a child cell is born with initial size $y:=|\Delta X_*(t)|$. 
After this branching event, the parent continues to evolve as $X_*$ and the child cell size proceeds independently of its parent, according to a process $X'_*$ of law $\Pm_y$. Furthermore, a new cell is generated at the first time $t'>t$ when $X_*$ has a jump such that 
$\frac{|\Delta X_*(t')|}{X_*(t'-)} >\e^{-\levOU}$ or $X'_*$ has a jump such that 
$\frac{|\Delta X'_*(t')|}{X'_*(t'-)} >\e^{-\levOU}$. This cell proceeds and produces offspring in a similar way, 
independently of the others. 
Iterating this argument, we produce all particles of $\cutOUd{\Xs}$. 
We hence conclude that $\cutOUd{\Xs}$ can be viewed as a truncated cell system in the sense of this section, associated with $X_*$, in which each child cell (together with its descendants) is killed whenever its size at birth is less than or equal to $\e^{-\levOU}$ times the size of the parent right before the birth of this child. Letting $\levOU \to \infty$, the claim follows from the monotonicity.
\end{proof}

\begin{cor}\label{Ch4:cor:binary}
Suppose that $\theta>0$. The law of a binary OU type growth-fragmentation $\Xsd$ is characterized by $(\kappa, \dOU)$. 
\end{cor}

\begin{proof}
Suppose that another binary OU type growth-fragmentation $\tilde{\Xs}$ also has index $\dOU$ and cumulant $\kappa$. 
Then it follows from Proposition~\ref{Ch4:prop:OU1} and Lemma~\ref{Ch4:lem:OU2} that $\tilde{\Xs}$ and $\Xsd$ have the same finite-dimensional distributions. Thus, the two processes $\tilde{\Xs}$ and $\Xsd$ have the same law because of the c\`adl\`ag property. 

Conversely, suppose that an OU type growth-fragmentation $\tilde{\Xs}$ have the same law as $\Xsd$.
Since it follows directly from Theorem~\ref{Ch4:thm:mean} that, for any $q>2$ such that $\kappa'(q)\ne 0$, there are the identities 
\[
\kappa(q) = \partial_t \log \mathbb{E} \left. \left[ \sum_{i=1}^{\infty} X_i(t)^q \right] \right|_{t=0}
\quad \text{and} \quad
\theta =-\frac{1}{\kappa'(q) q }\partial^2_{tt} \log \mathbb{E} \left. \left[ \sum_{i=1}^{\infty} X_i(t)^q \right] \right|_{t=0}, 
\]
we conclude that $\tilde{\Xs}$ and $\Xsd$ have the same index $\dOU$ and cumulant $\kappa$. 
\end{proof}

Conversely, each Markovian growth-fragmentation driven by an exponential OU type process is a binary OU type growth-fragmentation. 
\begin{prop}\label{Ch4:prop:MOU}
Let $Z$ be any OU type process with index $\theta>0$ and define its cumulant $\kappa$ by \eqref{Ch4:eq:kappa}. Then the Markovian growth-fragmentation $\Xsd := (X_1(t), X_2(t),\ldots)_{t\geq 0} $ associated with $\exp(Z)$ is a version of a binary OU type growth-fragmentation characterized by $(\kappa, \dOU)$. 
In particular, $\Xsd$ possesses a c\`adl\`ag version in $\co$ and for every $t\geq 0$ and $q\geq 2(1\vee \e^{\dOU t})$
$$\Exp{ \sum_{i=1}^{\infty}X_i(t)^q }  = \expp{ \int_{0}^t \kappa(q\e^{-\dOU s}) \dd s}<\infty.$$
\end{prop}

\begin{proof}
 Write $(\sigma, c, \Lambda,k, \dOU)$ for the characteristics of $Z$.
 Let $\nu_2$ be the image of $\Lambda$ by the map $z\mapsto (\max(\e^z, 1-\e^z), \min(\e^z, 1-\e^z), 0,\ldots)$, then $\nu_2 +k\delta_{\partialX}$ is a binary dislocation measure in the sense of Definition~\ref{Ch4:dfn:binary}, and thus there exists a binary OU type growth-fragmentation $\Xsd'$ with characteristics 
 \[
 \Big(\sigma,~ c-k+ \int_{(-\infty, -\log 2)} (1-2\e^{y}) \Lambda(\dd y),~ \nu_2 +k\delta_{\partialX}, \theta\Big).
 \] 
 A straightforward calculation shows that $\Xsd'$ has the same cumulant $\kappa$ as $Z$. 
Combining Proposition~\ref{Ch4:prop:OU1} and Lemma~\ref{Ch4:lem:OU2}, we deduce that $\Xsd'$ has the same finite-dimensional distributions as $\Xsd$.
We complete the proof by applying Theorem~\ref{Ch4:thm:mean} to $\Xsd'$. 
\end{proof}

\begin{rem}
Let $\tilde{X} $ be an OU type process with characteristics $(\tilde{\Phi},\tilde{\dOU})$, $\Xs $ and $\tilde{\Xs} $ be two Markovian growth-fragmentations driven respectively by $X$ and $\tilde{X}$. Then the following statements are equivalent: 
  \begin{enumerate}[label=(\roman*)]
  \item $\kappa = \tilde{\kappa}$ and $\dOU =\tilde{\dOU}$;
  \item $X$ and $\tilde{X}$ can be coupled to form a bifurcator that satisfies \ref{B1} and \ref{B2} in the proof of Lemma~\ref{Ch4:lem:OU2};
  \item the growth-fragmentations $\Xs $ and $\tilde{\Xs} $ have the same finite dimensional distributions.
  \end{enumerate}
  Indeed, we have already obtained ``$(i) \Rightarrow (ii) $'' and ``$(ii) \Rightarrow (iii) $'' from the proof of Lemma~\ref{Ch4:lem:OU2}. The implication ``$(iii) \Rightarrow (i)$'' follows from Proposition~\ref{Ch4:prop:MOU}. 
 This is an analogous result of Theorem~1.1 (for homogeneous growth-fragmentations) and Theorem~1.2 (for self-similar growth-fragmentations) in \cite{S:2}. 
\end{rem}

\begin{rem}\label{rem:Mgf-o}
When $\theta<0$, the function $F$ as in Lemma~\ref{Ch4:lem:OU} is not well-defined in general, unless $2\e^{\theta r} \in \dom$ for all $r>0$. 
For this case, unfortunately it seems difficult, if at all possible, to find a time-dependent excessive funtion $F'$ in the sense of \cite{S:2}, such that \eqref{Ch4:eq:H1OU} and \eqref{Ch4:eq:HetaOU} hold, and that 
\[
\inf_{r<l, x>a} F'(r,x)>0,	\qquad \text{for every} \quad a,l>0. 
\]
So we cannot apply \cite[Theorem~3.7]{S:2} to proceed the proof of Lemma~\ref{Ch4:lem:OU2} for the outward case. 
However, even without having such a function $F'$, we should still be able to prove Lemma~\ref{Ch4:lem:OU2} and Proposition~\ref{Ch4:prop:MOU} by using a direct approach similar to that in \cite[Propsition~2.15]{S:2}. Roughly speaking, this is done by changing the genealogy in the cell system.  
\end{rem}	

\section{A connection with random recursive trees}\label{Ch4:sec:RRT}
In this section we lift from \cite{BertoinBaur} a certain OU type growth-fragmentation that appears in the destruction of an infinite recursive tree. See also \cite{Moehle:MittagLeffler} for a related work.

An \emph{infinite recursive tree} is a random rooted tree with vertices indexed by $\N$, constructed recursively in the following way. We start with linking the vertex $1$ (the root) to the vertex $2$ by an edge denoted by $e_2$. Then we proceed by induction. For $i\geq 2$, vertex $i$ attaches to a vertex chosen uniformly from $\{1, \ldots, i-1\}$, say $j$, by an edge $e_{i}$. 

We destroy the infinite recursive tree by associating each $e_i$ with an independent exponential clock and breaking each edge when its clock rings. Then the vertices of this tree split into different connected clusters. Let $\Pi(t)= (\Pi_1(t),\Pi_2(t), \ldots)$ be the resulting partition of $\N$ at time $t\geq 0$, such that each $\Pi_i(t)$ is the set of the vertices of a cluster at time $t$, and they are listed in increasing order of the smallest element of the cluster. It has been proven in \cite{BertoinBaur} that
\[
W_i(t):= \lim_{n\to \infty} n^{-\e^{-t}} \#\{ k\leq n ~:~ k\in \Pi_i(t)\} \quad \text{exists for every } i\in \N.
\] 
Furthermore, $(W_i(t), i\in \N)$ can be rearranged in decreasing order, which produces a sequence denoted by $\Xs^{R}(t)$. Partial results of Proposition~2.3 and Theorem~3.1 in \cite{BertoinBaur} can be rewritten in our terms as follows.

\begin{prop}[\cite{BertoinBaur}]\label{Ch4:prop:RRT}
The process $\Xs^{R}$ is a binary OU type growth-fragmentation with characteristics $(\kappa_{R}, 1)$ in the sense of Corollary \ref{Ch4:cor:binary}, where 
\begin{equation*}
  \kappa_{R}(q) =q \psi(q+1) +(q-1)^{-1},\quad q>1, 
\end{equation*}
with $\psi$ denoting the digamma function, that is the logarithmic derivative of the gamma function. 
 Equivalently, $\Xs^{R}$ has characteristics $(0,-\gamma + 2\log 2, \nu,1)$, where $\gamma = 0.57721\ldots$ is the Euler-Mascheroni constant, and the dislocation measure $\nu$ is binary in the sense of Definition~\ref{Ch4:dfn:binary}, specified by 
\[
\nu(\dd s_1) = \left(s_1^{-2} + (1-s_1)^{-2}\right) \dd s_1, \qquad \frac{1}{2}\leq s_1 < 1.
\]
\end{prop}

Then by Proposition~\ref{Ch4:prop:branching} and Theorem~\ref{Ch4:thm:mean}, we recover immediately Theorem~3.4 in \cite{BertoinBaur}, which states the Markov property of $\Xs^{R}$ and that for every $t\geq 0$ and $q> \e^{t}$, there is 
\begin{equation}\label{Ch4:eq:RRT}
  \Exp{\sum_{i=1}^{\infty} X^R_i(t)^q} = \frac{q-1}{\e^{-t}q - 1} \frac{\Gamma(q)}{\Gamma(\e^{-t}q)}.
\end{equation}
Indeed, by the property of the digamma function $\psi$, an easy calculation shows that 
\[\expp{\int_0^t \kappa_R(\e^{-s}q ) \dd s} =\frac{\Gamma(q+1)}{\Gamma(\e^{-t} q+1)}\frac{q-1}{\e^{-t} q-1} \frac{\e^{-t} q}{q} = \frac{q-1}{\e^{-t}q - 1} \frac{\Gamma(q)}{\Gamma(\e^{-t}q)}.\]
Then \eqref{Ch4:eq:RRT} follows from Theorem~\ref{Ch4:thm:mean}. 

 For the readers' convenience, let us briefly justify Proposition~\ref{Ch4:prop:RRT} by using results in \cite{BertoinBaur}.
\begin{proof}[Proof of Proposition~\ref{Ch4:prop:RRT}]
Let $\xi$ be a spectrally negative L\'evy process with characteristics $(0,-\gamma+1, \Lambda,0)$, where the L\'evy measure $\Lambda$ has density 
\[\Lambda(\dd z) = \e^{z} (1-\e^z)^{-2} \dd z, \quad z \in (-\infty, 0).\]
We know from \cite{BertoinBaur} that the Laplace exponent of $\xi$ is $\Phi_R (q) := q \psi(q+1)$. \footnote{The L\'evy-Khintchine formula in \cite{BertoinBaur} has a compensation term different from \eqref{Ch4:eq:Phi}, so the drift coefficient is changed. } We also have that 
\[
 \int_{-\infty}^0 (1-\e^z)^q\e^z(1-\e^z)^{-2}\dd z = \frac{1}{q-1}, \qquad q>1.
\]
So $\xi$ has cumulant $\kappa_{R}$. 

 Write $P_x$ for the law of an exponential OU type process $X$ with characteristics $(\Phi_R, 1)$ starting from $x>0$, then we shall prove that $\Xs^R$ is Markovian growth-fragmentation associated with $X$. In this direction, let us consider a cell system $\Xc$ described as follows.  
Set the Eve process $\Xc_{\emptyset}:=W_1$, the weight process of the cluster $\Pi_1$ (that contains the root $1$). Then $\Xc_{\emptyset}$ has distribution $P_1$ by Theorem~3.1 in \cite{BertoinBaur}. At each jump time of $\Xc_{\emptyset}$, say $s>0$, the partition process $\Pi$ has a dislocation in which the block $\Pi_1(s)$ splits into $B_{1}$ and $B_{2}$, with $B_{1}$ being the block that contains $1$. 
Let $y:= \lim_{n\to \infty} n^{-\e^{-s}} \#\{ i\leq n ~:~ i\in B_2\}$, then we deduce by \cite[Proposition~2.3 and Theorem~3.1]{BertoinBaur} that the weight process 
\[
W^{B_{2}}_1(t):= \lim_{n\to \infty} n^{-\e^{-(t+s)}} \#\{ i\leq n ~:~ i\in \Pi_1(t+s)\cap B_2\}, \qquad t\geq 0
\] 
has conditional distribution $P_y$ given $\Xc_{\emptyset}$. We thus view $W^{B_{2}}_1$ as the daughter process born at the jump time $s$ of $\Xc_{\emptyset}$. In this way we associate each jump time of $\Xc_{\emptyset}$ with a daughter; these daughters are independent one of the others, and form the first generation of the cell system. By iteration of this argument, we obtain a cell system driven by $X$ and hence deduce that $\Xs^R$ is a Markovian growth-fragmentation associated with $X$. So we know from Proposition \ref{Ch4:prop:MOU} that $\Xs^R$ is a binary OU type growth-fragmentation process with characteristics $(\kappa_{R}, 1)$.
\end{proof}

\appendix
\section{Proofs of Lemma~\ref{Ch4:lem:cut} and Lemma~\ref{Ch4:lem:Wselect}}\label{sec:A}

 \begin{proof}[Proof of Lemma~\ref{Ch4:lem:cut}]
 	The proof is an adaptation of the arguments in \cite[Lemma~3]{Bertoin:CF}.
We shall check that $\cutOU{\Zs}$ fulfills Definition~\ref{Ch4:dfn:BOUPF} with a different genealogy. 
 
For every $i\in \N$, let $\bar{1}_{i}:= (1,1,\ldots,1)\in \N^i$, with $\bar{1}_{0}= \emptyset$ by convention.  
With notation in Definition \ref{Ch4:dfn:BOUPF}, we write $\rr_i:= \Delta a_{\bar{1}_i}$ for every $i\in \N$ and derive $\cutOU{\rr_{i}}$ from $\rr_i$ by \eqref{Ch4:eq:cut}. 
As $\rr_{i}$ has the law of $\mu(\cdot \mid \Rdtwo)$, we easily deduce that $\Prob{\cutOU{\rr_{i}}\not\in \Rd_1} = \frac{\cutOU{\mu}(\Rdtwo)}{\mu(\Rdtwo)}$. 
Let $N:=\inf \{i\ge 0:\cutOU{\rr_{i}}\not\in \Rd_1 \}$, then for each $i\leq N-1$, only the closest child of $\bar{1}_i$ is kept in the truncated system $\cutOU{\Zs}$, but the other children are all killed. 
Therefore, at any time before  $b_{\bar{1}_{N}}$, in $\cutOU{\Zs}$ there is only one particle, which shall be viewed as the ancestor $\emptyset$ in the truncated system $\cutOU{\Zs}$. 
At its lifetime $\cutOU{\lambda}_{\emptyset} := b_{\bar{1}_{N}}$, it splits into more than one particles, located at $Z_{\emptyset} + \cutOU{\rr_{N}}$. So we define $\cutOU{\Delta a}_{\emptyset}:=\cutOU{\rr_{N}}$, which is a random variable of law $\cutOU{\mu}(\cdot \mid\Rdtwo)$. 
Since $N$ has the geometric distribution with parameter $\frac{\cutOU{\mu}(\Rdtwo)}{\mu(\Rdtwo)}$, from basic property of exponential random variables, we know that $\cutOU{\lambda}_{\emptyset}$ has the exponential distribution with parameter $\cutOU{\mu}(\Rdtwo)$. 

We next investigate the movement $\cutOU{Z}_{\emptyset}$ of the ancestor $\emptyset$.
Write recursively a sequence $(\tilde{a}_{\bar{1}_{j}})_{j\geq 0}$ such that $\tilde{a}_{\bar{1}_{0}}:=0$ and
$\tilde{a}_{\bar{1}_{j+1}} := \e^{-\dOU \lambda_{\bar{1}_{j}} }\tilde{a}_{\bar{1}_{j}} + Z_{\bar{1}_j}(\lambda_{\bar{1}_j}).$
Then we define a process $Z_{\bar{1}}$ by 
\begin{equation*}
  Z_{\bar{1}}(t):= \e^{-\dOU (t-b_{\bar{1}_j})} \tilde{a}_{\bar{1}_{j}} +Z_{\bar{1}_j}(t- b_{\bar{1}_j}), \qquad \text{for } t\in [ b_{\bar{1}_j} , b_{\bar{1}_j}+\lambda_{\bar{1}_j}) \text{ with }j\geq 0.
\end{equation*}
It follows from the simple Markov property that $Z_{\bar{1}}$ is an OU type process with characteristics $(\psi, \dOU)$. 
We also define a process
\begin{equation*}
  \cutOU{\eta}_{\bar{1}}(t):= \sum_{i = 0}^{j} \e^{-\dOU (t-b_{\bar{1}_i}) }\Delta a_{\bar{1}_i} \ind{\cutOU{\rr_{i}}\in \Rd_1}, \qquad \text{for } t\in [ b_{\bar{1}_j} , b_{\bar{1}_j}+\lambda_{\bar{1}_j}) \text{ with }j\geq 0.
\end{equation*}
 By \eqref{Ch4:eq:solOU}, it is an OU type process associated with a compound Poisson process on $(-\infty, 0)$ with L\'evy measure 
\[\mu(r_1 \in \dd z: ~ \cutOU{\rr}\in \Rd_1, \rr \not\in \Rd_1), \quad z\in (-\infty, 0).\] 
Since $\cutOU{Z}_{\emptyset}$ is the superposition of the two independent processes  
 $Z_{\bar{1}}$  and $\cutOU{\eta}_{\bar{1}}$, we have by Lemma~\ref{Ch4:lem:addOU} that $\cutOU{Z}_{\emptyset}$ is an OU type process with characteristics $(\cutOU{\psi}, \dOU)$, where 
\begin{equation}\label{eq:psil}
    \cutOU{\psi}(q) := \psi(q) + \int_{(-\infty, 0)} (\e^{q z} -1) \mu(r_1 \in \dd z: ~ \cutOU{\rr}\in \Rd_1, \rr \not\in \Rd_1), \qquad q\geq 0.
 \end{equation}
Using the fact that 
\[\int_{\Rd} (1-\e^{r_1})\cutOU{\mu}(\dd \rr) =\int_{\Rd} (1-\e^{r_1})\mu(\dd \rr)\]
and that $\rr\in \Rd_1$ infers $\cutOU{\rr}\in \Rd_1$, we deduce an identity 
\begin{equation*}
\cutOU{\psi}(q) = \frac{1}{2} \sigma^2 q^2 + \left( c + \int_{\Rdtwo} (1-\e^{r_1})\cutOU{\mu}(\dd \rr) \right) q 
 + \int_{\Rd_1}  \left(\e^{qr_1}-1 +  q(1-\e^{r_1}) \right) \cutOU{\mu}(\dd \rr).
 \end{equation*}
By iterating this argument and comparing with Definition~\ref{Ch4:dfn:BOUPF}, we complete the proof.  
 \end{proof}
 
  \begin{proof}[Proof of Lemma~\ref{Ch4:lem:Wselect}]
 From the proof of Lemma~\ref{Ch4:lem:cut}, we readily know that the selected atom $Z_*$ is an OU type process with characteristics $(\cutOU[0]{\psi}, \dOU)$, with $\cutOU[0]{\psi}$ given by \eqref{eq:psil}. One easily checks that $\cutOU[0]{\psi}= \Phi_*$. 
\end{proof}

\section{Proof of Equation~(\ref{eq:gfe-unique}) }\label{sec:B}
It suffices to prove for the case $0\le s+ t\le T_0$. 
For simplicity, write 
\[
C(t, y):=  \e^{-\alpha \e^{\theta s} z} y^{\alpha \e^{\theta(s+t)}} \exp\Big(-\int_s^{s+t} \kappa(\alpha \e^{\theta r}) \dd r\Big) . 
\]
Since $\rho'_{\e^z}$ is a solution to \eqref{Ch4:eq:gf}, we have
\[ 
\prm{ P'_{s,s+t}(z,\cdot)}{ g(s+t,\cdot)} 
=  \prm{ \rho'_{\e^z}(t,\dd y)}{ C(t,y) g(s+t, \log y)} 
= g(s,z) + \int_0^t C(r,y) \prm{ \rho'_{\e^z}(r,\dd y)}{L(s+r,y)} \dd r,  
\] 
where $L(s+r,y):= L_1+ L_2+ \frac{1}{2} \sigma^2 L_3 + \big(c+\frac{1}{2} \sigma^2  - \theta \log y\big)L_4 + L_5$ and 
\begin{align*}
L_1&=  \alpha \e^{\theta (s+r)} \theta \log y~ g(s+r, \log y) + \partial_t g(s+r, \log y)  ,\\
L_2&= -\kappa(\alpha \e^{\theta(s+r)} ) g(s+r, \log y)\\
&= -\bigg( \frac{1}{2} \sigma^2 (\alpha \e^{\theta (s+r)})^2 + \alpha \e^{\theta (s+r)} c+ \int_{\Sd} \Big( \sum_{i=1}^{\infty} s_i^{\alpha \e^{\theta (s+r)}} - 1+ (1-s_1) \alpha \e^{\theta (s+r)} \Big) \nu(\dd \sd) \bigg) g(s+r, \log y)
, \\
L_3&=   \alpha \e^{\theta (s+r)}(\alpha \e^{\theta (s+r)}-1) g(s+r, \log y)  \\
	&  \qquad \qquad +  \alpha \e^{\theta (s+r)}\partial_x g(s+r, \log y)  + (\alpha \e^{\theta (s+r)}-1) \partial_x g(s+r, \log y) +  \partial^2_{xx} g(s+r, \log y),\\
L_4&=  \alpha \e^{\theta (s+r)}g(s+r, \log y)  + \partial_x g(s+r, \log y), \\
L_5 &= \int_{\Sd} \Big( \sum_{i=1}^{\infty} s_i^{\alpha \e^{\theta (s+r)}} g(s+r, \log y +\log s_i) - g(s+r, \log y ) + (1-s_1) L_4 \Big) \nu(\dd \sd). 
\end{align*}
On the other hand, let us write $\mathcal{A} g(s+r, \log y)= A_1 +A_2 +A_3$, where
 \begin{align*}
 A_1&= \partial_t g(s+r, \log y) + \frac{1}{2} \sigma^2 \partial^2_{xx} g(s+r, \log y)\\
A_2&= \Big(c+ \sigma^2 \alpha \e^{\theta (s+r)} - \theta \log y + \int_{\Sd} \big( (1-s_1) -\sum_{i=1}^{\infty} s_i^{\alpha \e^{\theta (s+r)}}(1-s_i) \big ) \nu(\dd \sd)\Big) \partial_{x} g(s+r, \log y) \\
A_3 &=  \int_{\Sd} \Big( \sum_{i=1}^{\infty} s_i^{\alpha \e^{\theta (s+r)}} g(s+r, \log y +\log s_i) \\
&\qquad \qquad  -\sum_{i=1}^{\infty} s_i^{\alpha \e^{\theta (s+r)}} g(s+r, \log y )  
  + \sum_{i=1}^{\infty} s_i^{\alpha \e^{\theta (s+r)}}(1-s_i) \partial_{x} g(s+r, \log y)\Big) \nu(\dd \sd). 
\end{align*}
Comparing these terms, we deduce the identity $L(s+r,y)= \mathcal{A} g(s+r, \log y)$. 
This entails that 
\begin{equation*}
C(t,y)  \prm{ \rho'_{\e^z}(r,\dd y)}{L(s+r,y) }=\prm{ P'_{s,s+r}(z,\cdot)}{ \mathcal{A} g(s+r,\cdot)} , 
\end{equation*}
which ends the proof. 

\section*{Acknowledgement}
The author is very grateful to Jean Bertoin for his guidance and advice throughout this research. The author also warmly thanks Igor Kortchemski, Robin Stephenson, Matthias Winkel and an anonymous referee for their insightful comments which helped to improve the work considerably. This research was partially supported by the p\^ole mathSTIC of Universit\'e Paris 13 and SNSF fellowship P2ZHP2\_171955. 

\bibliographystyle{abbrv}
\bibliography{quanshi}

\begin{thebibliography}{10}

\bibitem{AdamczakMilos:CLT}
R.~Adamczak and P.~Mi{\l}o{\'s}.
\newblock C{LT} for {O}rnstein-{U}hlenbeck branching particle system.
\newblock {\em Electron. J. Probab.}, 20:no. 42, 35, 2015.

\bibitem{Applebaum:Levy}
D.~Applebaum.
\newblock {\em L\'evy processes and stochastic calculus}, volume 116 of {\em
  Cambridge Studies in Advanced Mathematics}.
\newblock Cambridge University Press, Cambridge, second edition, 2009.

\bibitem{AthreyaNey}
K.~B. Athreya and P.~E. Ney.
\newblock {\em Branching processes}.
\newblock Springer-Verlag, New York-Heidelberg, 1972.
\newblock Die Grundlehren der mathematischen Wissenschaften, Band 196.

\bibitem{NielsenShephard:1}
O.~E. Barndorff-Nielsen and N.~Shephard.
\newblock Non-{G}aussian {O}rnstein-{U}hlenbeck-based models and some of their
  uses in financial economics.
\newblock {\em J. R. Stat. Soc. Ser. B Stat. Methodol.}, 63(2):167--241, 2001.

\bibitem{NielsenShephard:2}
O.~E. Barndorff-Nielsen and N.~Shephard.
\newblock Econometric analysis of realized volatility and its use in estimating
  stochastic volatility models.
\newblock {\em J. R. Stat. Soc. Ser. B Stat. Methodol.}, 64(2):253--280, 2002.

\bibitem{BertoinBaur}
E.~Baur and J.~Bertoin.
\newblock The fragmentation process of an infinite recursive tree and
  {O}rnstein-{U}hlenbeck type processes.
\newblock {\em Electron. J. Probab.}, 20:no. 98, 1--20, 2015.

\bibitem{Berestycki:Frag-Ranked}
J.~Berestycki.
\newblock Ranked fragmentations.
\newblock {\em ESAIM Probab. Statist.}, 6:157--175 (electronic), 2002.

\bibitem{Bertoin:Levy}
J.~Bertoin.
\newblock {\em L\'evy processes}, volume 121 of {\em Cambridge Tracts in
  Mathematics}.
\newblock Cambridge University Press, Cambridge, 1996.

\bibitem{Bertoin:Frag-Book}
J.~Bertoin.
\newblock {\em Random fragmentation and coagulation processes}, volume 102 of
  {\em Cambridge Studies in Advanced Mathematics}.
\newblock Cambridge University Press, Cambridge, 2006.

\bibitem{Bertoin:CF}
J.~Bertoin.
\newblock Compensated fragmentation processes and limits of dilated
  fragmentations.
\newblock {\em Ann. Probab.}, 44(2):1254--1284, 2016.

\bibitem{Bertoin:GF-Markovian}
J.~Bertoin.
\newblock Markovian growth-fragmentation processes.
\newblock {\em Bernoulli}, 23(2):1082--1101, 2017.

\bibitem{BCK:Martingale}
J.~Bertoin, T.~Budd, N.~Curien, and I.~Kortchemski.
\newblock Martingales in self-similar growth-fragmentations and their
  connections with random planar maps.
\newblock {\em Probab. Theory Related Fields}, 2018+.
\newblock To appear.

\bibitem{BertoinGnedin}
J.~Bertoin and A.~V. Gnedin.
\newblock Asymptotic laws for nonconservative self-similar fragmentations.
\newblock {\em Electron. J. Probab.}, 9:no. 19, 575--593, 2004.

\bibitem{BM-blp}
J.~Bertoin and B.~Mallein.
\newblock Infinitely ramified point measures and branching {L}{\'e}vy
  processes.
\newblock {\em Ann. Probab.}, 2018+.
\newblock To appear.

\bibitem{BertoinRouault}
J.~Bertoin and A.~Rouault.
\newblock Discretization methods for homogeneous fragmentations.
\newblock {\em J. London Math. Soc. (2)}, 72(1):91--109, 2005.

\bibitem{BertoinStephenson}
J.~Bertoin and R.~Stephenson.
\newblock Local explosion in self-similar growth-fragmentation processes.
\newblock {\em Electron. Commun. Probab.}, 21:Paper No. 66, 12, 2016.

\bibitem{BertoinWatson}
J.~Bertoin and A.~R. Watson.
\newblock Probabilistic aspects of critical growth-fragmentation equations.
\newblock {\em Adv. in Appl. Probab.}, 48(A):37--61, 2016.

\bibitem{BertoinWatson2}
J.~Bertoin and A.~R. Watson.
\newblock A probabilistic approach to spectral analysis of growth-fragmentation
  equations.
\newblock {\em J. Funct. Anal.}, 274(8):2163--2204, 2018.

\bibitem{Biggins:1977}
J.~D. Biggins.
\newblock Martingale convergence in the branching random walk.
\newblock {\em J. Appl. Probability}, 14(1):25--37, 1977.

\bibitem{Biggins:1992}
J.~D. Biggins.
\newblock Uniform convergence of martingales in the branching random walk.
\newblock {\em Ann. Probab.}, 20(1):137--151, 1992.

\bibitem{Billingsley:Convergence}
P.~Billingsley.
\newblock {\em Convergence of probability measures}.
\newblock Wiley Series in Probability and Statistics: Probability and
  Statistics. John Wiley \& Sons, Inc., New York, second edition, 1999.
\newblock A Wiley-Interscience Publication.

\bibitem{Dadoun:Asymptotics}
B.~Dadoun.
\newblock Asymptotics of self-similar growth-fragmentation processes.
\newblock {\em Electron. J. Probab.}, 22:Paper No. 27, 30, 2017.

\bibitem{DoumicEscobedo}
M.~Doumic and M.~Escobedo.
\newblock Time asymptotics for a critical case in fragmentation and
  growth-fragmentation equations.
\newblock {\em Kinet. Relat. Models}, 9(2):251--297, 2016.

\bibitem{Doumic}
M.~Doumic, M.~Hoffmann, N.~Krell, and L.~Robert.
\newblock Statistical estimation of a growth-fragmentation model observed on a
  genealogical tree.
\newblock {\em Bernoulli}, 21(3):1760--1799, 2015.

\bibitem{DFS:ap}
D.~Duffie, D.~Filipovi\'c, and W.~Schachermayer.
\newblock Affine processes and applications in finance.
\newblock {\em Ann. Appl. Probab.}, 13(3):984--1053, 2003.

\bibitem{EHK:SLLN}
J.~Engl{\"a}nder, S.~C. Harris, and A.~E. Kyprianou.
\newblock Strong law of large numbers for branching diffusions.
\newblock {\em Ann. Inst. Henri Poincar\'e Probab. Stat.}, 46(1):279--298,
  2010.

\bibitem{Fil:tiap}
D.~Filipovi\'c.
\newblock Time-inhomogeneous affine processes.
\newblock {\em Stochastic Process. Appl.}, 115(4):639--659, 2005.

\bibitem{Jagers}
P.~Jagers.
\newblock General branching processes as {M}arkov fields.
\newblock {\em Stochastic Process. Appl.}, 32(2):183--212, 1989.

\bibitem{Jakubowski}
A.~Jakubowski.
\newblock Convergence in various topologies for stochastic integrals driven by
  semimartingales.
\newblock {\em Ann. Probab.}, 24(4):2141--2153, 1996.

\bibitem{JMP:cv}
A.~Jakubowski, J.~M{\'e}min, and G.~Pag{\`e}s.
\newblock Convergence en loi des suites d'int\'egrales stochastiques sur
  l'espace {${\bf D}^1$} de {S}korokhod.
\newblock {\em Probab. Theory Related Fields}, 81(1):111--137, 1989.

\bibitem{Kallenberg}
O.~Kallenberg.
\newblock {\em Foundations of modern probability}.
\newblock Probability and its Applications (New York). Springer-Verlag, New
  York, second edition, 2002.

\bibitem{Lansky:OU-neuronal}
P.~L{\'a}nsk{\'y} and L.~Sacerdote.
\newblock The {O}rnstein-{U}hlenbeck neuronal model with signal-dependent
  noise.
\newblock {\em Phys. Lett. A}, 285(3-4):132--140, 2001.

\bibitem{LiebLoss:Analysis}
E.~H. Lieb and M.~Loss.
\newblock {\em Analysis}, volume~14 of {\em Graduate Studies in Mathematics}.
\newblock American Mathematical Society, Providence, RI, second edition, 2001.

\bibitem{Spine}
R.~Lyons, R.~Pemantle, and Y.~Peres.
\newblock Conceptual proofs of {$L\log L$} criteria for mean behavior of
  branching processes.
\newblock {\em Ann. Probab.}, 23(3):1125--1138, 1995.

\bibitem{MischlerScher}
S.~Mischler and J.~Scher.
\newblock Spectral analysis of semigroups and growth-fragmentation equations.
\newblock {\em Ann. Inst. H. Poincar\'e Anal. Non Lin\'eaire}, 33(3):849--898,
  2016.

\bibitem{Moehle:MittagLeffler}
M.~M{\"o}hle.
\newblock The {M}ittag-{L}effler process and a scaling limit for the block
  counting process of the {B}olthausen-{S}znitman coalescent.
\newblock {\em ALEA Lat. Am. J. Probab. Math. Stat.}, 12(1):35--53, 2015.

\bibitem{nerman1981CMJ}
O.~Nerman.
\newblock On the convergence of supercritical general ({C}-{M}-{J}) branching
  processes.
\newblock {\em Z. Wahrsch. Verw. Gebiete}, 57(3):365--395, 1981.

\bibitem{Sato:Levy}
K.-i. Sato.
\newblock {\em L\'evy processes and infinitely divisible distributions},
  volume~68 of {\em Cambridge Studies in Advanced Mathematics}.
\newblock Cambridge University Press, Cambridge, 2013.
\newblock Translated from the 1990 Japanese original, Revised edition of the
  1999 English translation.

\bibitem{S:2}
Q.~Shi.
\newblock Growth-fragmentation processes and bifurcators.
\newblock {\em Electron. J. Probab.}, 22:Paper No. 15, 25, 2017.

\bibitem{SW-tilting}
Q.~Shi and A.~R. Watson.
\newblock Probability tilting of compensated fragmentations.
\newblock Preprint, \href{http://arxiv.org/abs/1707.00732v2}{1707.00732v2}
  [math.PR], 2017.

\bibitem{Shi:BRW-Book}
Z.~Shi.
\newblock {\em Branching random walks}, volume 2151 of {\em Lecture Notes in
  Mathematics}.
\newblock Springer, Cham, 2015.
\newblock Lecture notes from the 42nd Probability Summer School held in Saint
  Flour, 2012, {\'E}cole d'{\'E}t{\'e} de Probabilit{\'e}s de Saint-Flour.
  [Saint-Flour Probability Summer School].

\end{thebibliography}

\end{document}